\documentclass[reqno]{amsart}
\usepackage{enumerate}
\usepackage{amsfonts,amssymb,amsthm, amscd, amsmath, amsbsy, url,color}
\usepackage{cite}
\usepackage[hypertex]{hyperref}

\theoremstyle{plain}

\newtheorem{thm}{Theorem}[section]
\newtheorem{cor}{Corollary}[section]
\newtheorem{lem}{Lemma}[section]
\newtheorem{prop}{Proposition}[section]

\theoremstyle{remark}
\newtheorem{rem}{{\rm\bf Remark}}[section]

\numberwithin{equation}{section}

\newcommand{\ee}{\mathbb{E}}

\newcommand{\rr}{\mathbb{R}}
\newcommand{\pp}{\mathbb{P}}

\newcommand{\zz}{\mathbb{Z}}

\def\FF{\mathcal F}

\def\FF{\mathcal F}

\def\dsp{\displaystyle}

\begin{document}

\title[MDP for empirical covariance from a unit root]
{Moderate deviations principle for empirical covariance from a unit root}

\author[Y. Miao]{Yu Miao}
\address[Y. Miao]{College of Mathematics and Information Science, Henan Normal University, Henan Province, 453007, China.}
\email{\href{mailto: Y. Miao
<yumiao728@yahoo.com.cn>}{yumiao728@yahoo.com.cn}}

\author[Y. -L. Wang]{Yan-Ling Wang}
\address[Y. -L. Wang]{College of Mathematics and Information Science, Henan Normal University, Henan Province, 453007, China.}

\author[G. -Y. Yang]{Guang-yu Yang}
\address[G. -Y. Yang]{Department of Mathematics, Zhengzhou University, Henan Province, 450001, China.}
\email{\href{mailto: G. -Y. Yang
<study\_yang@yahoo.com.cn>}{study\_yang@yahoo.com.cn}}

\begin{abstract}
In the present paper, we consider the linear autoregressive model in
$\rr$,
$$ X_{k,n}=\theta_n X_{k,n-1}+\xi_k,\ \ k=0,1,\cdots,n,\ \ n\ge
1$$ where $\theta_n\in [0,1)$ is unknown, $(\xi_k)_{k\in\zz}$ is a
sequence of centered i.i.d. r.v. valued in $\rr$ representing the
noise. When $\theta_n\to 1$, the moderate deviations principle for
empirical covariance is discussed and as statistical applications we
provide the moderate deviation estimates of the least square and the
Yule-Walker estimators of the parameter $\theta_n$.

\end{abstract}



\maketitle



\section{Introduction}
There is a great deal of the econometric literature of the last 20
years which has focused on the issue of testing for the unit root
hypothesis in economic time series. Regression asymptotics with roots at or near unity have played an important role in time
series econometrics. This has been typically done by
using autoregressive models with fixed coefficients and then testing
for the autoregressive parameter being equal to 1 \cite{DF79, DF81}.
More recently, some attention has been dedicated to random
coefficient autoregressive models. This way of handling the data
allows for large shocks in the dynamic structure of the model, and
also for some flexibility in the features of the volatility of the
series, which are not available in fixed coefficient autoregressive
models.

In the present paper, we consider the following linear
autoregressive model in $\rr$,
 \begin{equation}\label{M}
  X_{k,n}=\theta_n X_{k-1,n}+\xi_k,\ \ k=0,1,\cdots,n,\ \ \ n\ge 1
 \end{equation}
where $\theta_n\in \Theta\subset\rr$ (the space of parameter) is
unknown, $(\xi_k)_{k\in\zz}$ is a sequence of centered i.i.d. r.v.
valued in $\rr$ representing the noise and which is independent of
$X_{0,n}$, and $(X_{k,n})_{0\le k\le n}$ is observed. For every
$n\ge 0$, assume that the law of $X_{0,n}$ is invariant (or
equivalently $(X_{k,n})_{0\le k\le n}$ is stationary), it is easy to
see a stationary solution to (\ref{M}), which is given by
$$
X_{k,n}=\sum_{p=0}^\infty\theta_n^p\xi_{k-p},\ \ k\ge 0
$$
only if $|\theta_n|< 1$.

It is not difficult to see that the linear autoregressive model (\ref{M}) is a special moving average process.
A
general moving average process is given by
$$
X_n:=\sum_{j=-\infty}^{+\infty}a_{j-n}\xi_j=\sum_{j=-\infty}^{+\infty}a_j\xi_{n+j},
\ \ \ \ \forall\ n\in \zz,
$$
where $(\xi_n)_{n\in\zz}$ is i.i.d., $(a_n)_{n\in \zz}$ is a
sequence of real numbers such that
$$
\sum_{n\in \zz}|a_n|^2<\infty.
$$

There are two important issues for the model (\ref{M}): (1) the
estimate of the covariance $Cov(X_{0,n},X_{l,n}):=\ee(X_{0,n}-\ee
X_{0,n})(X_{l,n}-\ee X_{l,n})$; (2) the estimate of $\theta_n$. The
most natural estimator of $Cov(X_{0,n},X_{l,n})$ ($l\ge 0$) is given
by the empirical covariance (with the given sample $(X_{k,n})_{0\le
k\le n-l}$)
 \begin{equation}\label{C}
 C^{*}_{l,n}=\frac{1}{n-l}\sum_{k=1}^{n-l}X_{k+l,n}X_{k,n}
 \end{equation}
and for estimating $\theta_n$, the following two estimators are widely
used:

(i) Least Square Estimator:
 \begin{equation}\label{LS}
  \hat\theta_{n}=\frac{\sum_{k=1}^nX_{k,n}X_{k-1,n}}{\sum_{k=1}^nX_{k-1,n}^2}.
 \end{equation}

 (ii) Yule-Walker Estimator:
 \begin{equation}\label{YW}
  \tilde\theta_{n}=\frac{\sum_{k=1}^nX_{k,n}X_{k-1,n}}{\sum_{k=0}^nX_{k,n}^2}.
 \end{equation}

In this paper, we are concerned with the moderate deviations
principle of the covariance estimation $ C^{*}_{l,n}$ and the
parameter estimators $\hat\theta_n, \tilde\theta_n$ for the linear
autoregressive model under the case: $\theta_n\in[0,1)$ and
$\theta_n\to 1$.

 The study on large deviations and moderate deviation are relatively recent and these works concentrate almost
 on the case of the fixed autoregressive coefficient $\theta_n\equiv\theta\in(-1,1)$, i.e.,
 \begin{equation}\label{M-1}
 X_n=\theta X_{n-1}+\xi_n, \ \ \ n\ge 0.
 \end{equation}
  For the Gaussian case (i.e., the noise $\xi$ is assumed Gaussian), this
subject is opened by Donsker and Varadhan \cite{DV2} who proved the
level-$3$ large deviation principle (the definition of large
deviations of level-$3$ could be found in \cite{Ellis}) for general
stationary Gaussian processes under the continuity of the spectral
function. Bryc and Dembo \cite{BD1} proved for the first the large
and moderate deviation principles for the empirical variance
$C^*_{0,n}(=n^{-1}\sum_{k=1}^n X_k^2)$ even for general stationary
Gaussian processes. Bercu et al. \cite{BGR} proved the large
deviation principle for $C^*_{l,n}(=n^{-1}\sum_{k=1}^n X_{k+l}X_k)$,
$l\ge 0$ (which is much more delicate than $C^*_{0,n}$) and for
$\hat\theta_{n}$, $\tilde\theta_{n}$.

For the Non-Gaussian case, Wu \cite{Wu04} first extended Donsker-Varadhan's theorem
on large deviations of level-$3$ from stationary Gaussian processes to
general moving average processes under the Gaussian integrability
condition on the driven variable $\xi$. Djellout et al. \cite{DGW06}
established, in the one-dimensional case, moderate deviation
principle for non-linear functionals of general moving average
processes covering the case of $C_{l,n}^*$ and for the periodogram,
but under the assumption that the law of the driven random  variable
$\xi$ satisfies the log-Sobolev inequality, stronger than the
Gaussian integrability in \cite{Wu04}.

 For the case of Hilbertian autoregressive model with driven random variable
$\xi$ satisfying the Gaussian integrability condition, in which
$\{\xi_k, X_k\}_{k\in\zz}$ take values in some separable Hilbert
space $H$, Mas and Menneteau \cite{MaMl} established large and
moderate deviation for the empirical mean
$\overline{X}_n=\frac{1}{n}\sum_{k=1}^n X_k$, and moderate
deviation for the empirical variance matrix
$\frac{1}{n}\sum_{k=1}^{n}X_k\otimes X_k$, where $x\otimes y$
($x,y\in H$) denotes the linear operator from $H$ to $H$,
$$
 x\otimes y: h\in H\to \langle x, h\rangle y,
$$
extending the result of Bryc-Dembo \cite{BD1} from $\rr^d$ to $H$,
and especially from Gaussian case to general sub-Gaussian case.
Furthermore, Menneteau \cite{Ml} obtained some laws of the iterated
logarithm in Hilbertian autoregressive models for the empirical
covariance $\frac{1}{n}\sum_{k=1}^{n}X_k\otimes X_k$. Recently, Miao
and Shen \cite{MS} obtained a moderate deviations principle for
$C_{n,l}^*$ of the autoregressive process (\ref{M-1}), which
removed the assumption of log-Sobolev inequality on the driven
variable in \cite{DGW06}, for the particular but important
auto-regression model. In addition, they provided the moderate
deviation estimates of the least squares and the Yule-Walker
estimators of the unknown parameter of an autoregressive process. In
\cite{M1}, the author also considered the discounted large deviation
principle for the autoregressive processes (\ref{M-1}).

Our main purpose in the paper is to extend the moderate deviations
principle for the empirical covariance from the case
$\theta_n=\theta$ to the case $ \theta_n\to 1$. The method of proof
relies mainly on a moderate deviation for triangular arrays of
finitely-dependent sequences and the exponential approximation. This
paper is organized as follows. The next section is devoted to the
descriptions of our main results and their statistical applications.
In Section 3, we give some preparations and develop a new moderate
deviation for $m$-dependent sequence with unbounded $m$. The proofs
of main results are obtained in the remaining sections.

\section{Main results}

\subsection{Assumptions}

Let $\{\xi_n\}_{n\in\zz}$ be a sequence of real valued centered
i.i.d. random variables, and suppose that the following conditions
hold:
\begin{enumerate}
\item  the unknown parameter $\theta_n$ satisfies $\theta_n\in[0,1)$, $\theta_n\to 1$;

 \item $\ee\xi_0=0$ and $\xi_0$ satisfies the Gaussian integrability condition, i.e.,
 there exists $\alpha>0$, such that
 $$
 \ee e^{\alpha \xi_0^2}<\infty;
 $$

\item the moderate deviation scale $(b_n)$ is a sequence of positive numbers satisfying
$$
 b_n\to\infty,\ \ \ \frac{\sqrt{n}(1-\theta_n)^2}{b_n}\to\infty.
$$

\end{enumerate}

Here we need to note that the condition (3) implies
$$
\dsp\lim_{n\to\infty}n(1-\theta_n)=\infty,\ \ \text{and} \ \
\frac{\sqrt n}{b_n}\to \infty.
$$
\subsection{Moderate deviations principle}
The following is our main theorem.
\begin{thm}\label{mdp-thm1}
Assume that the conditions (1), (2) and (3) are satisfied and let
$M$ be a non-negative integer, then for all $r>0$, when $0\le l\le
M$, we have
 \begin{equation}\label{mdp-thm-11}
 \lim_{n\to\infty}\frac{1}{b_n^2}\log\pp\left(\frac{(1-\theta_n^2)^{3/2}\sqrt{n}}{b_n}\left|C^{*}_{l,n}-\ee C_{l,n}^{*}\right|\ge r\right)
 =-\frac{r^2}{8(\ee\xi_0^2)^2}.
 \end{equation}
\end{thm}

\begin{rem}
Since $M$ is fixed and $l$ is finite, then the form
(\ref{mdp-thm-11}) is equivalent to
\begin{equation}\label{mdp-thm-12}
\lim_{n\to\infty}\frac{1}{b_n^2}\log\pp\left(\frac{(1-\theta_n^2)^{3/2}\sqrt{n-l}}{b_n}\left|C^{*}_{l,n}-\ee
C_{l,n}^{*}\right|\ge r\right)
 =-\frac{r^2}{8(\ee\xi_0^2)^2}.
 \end{equation}
In the process of proving Theorem \ref{mdp-thm1}, we often use the
form (\ref{mdp-thm-12}) in order to avoid extra explanation.
\end{rem}

The following result supplies a moderate deviation for the linear
combination of the empirical covariance. For the case that the
unknown parameter $\theta_n$ is fixed ($\theta_n=\theta$), we can
succeed in obtaining the moderate deviation of the parameter
estimators $\hat\theta_n, \tilde\theta_n$ by utilizing the following
result directly.

\begin{thm}\label{mdp-thm2} Under the conditions in Theorem \ref{mdp-thm1}, for any $n$ and $0\le
l\le M$, let $\{a_{l,n}\}$ be a sequences of real numbers with
$\lim\limits_{n\to\infty}a_{l,n}=a_l$, and assume that $a_l\ne 0$
for some $0\le l\le M$. Then for any $r>0$, we have
$$
\lim_{n\to\infty}\frac{1}{b_n^2}\log\pp\left(\frac{(1-\theta_n^2)^{3/2}\sqrt{n}}{b_n}\left|\sum_{l=0}^Ma_{l,n}(C^{*}_{l,n}-\ee
C_{l,n}^{*})\right|\ge r\right)
 =-\frac{r^2}{2\Sigma^2}
$$
where
$$
\Sigma^2=4\left(\sum_{j=0}^Ma_j\right)^2(\ee\xi_0^2)^2.
$$
\end{thm}

\begin{rem} Under the conditions of Theorem \ref{mdp-thm2}, if
$\sum_{j=0}^Ma_j=1$, then we have
$$
\lim_{n\to\infty}\frac{1}{b_n^2}\log\pp\left(\frac{(1-\theta_n^2)^{3/2}\sqrt{n}}{b_n}\left|\sum_{l=0}^Ma_{l,n}(C^{*}_{l,n}-\ee
C_{l,n}^{*})\right|\ge r\right)
 =-\frac{r^2}{8(\ee\xi_0^2)^2}.
$$
In particular, Theorem \ref{mdp-thm1} holds, if there exists some
$0<l<M$, such that
$$
a_k=\begin{cases} 1,& \ k=l\\
0, &\ k\ne l
\end{cases}.
$$

\end{rem}

\subsection{Applications}
In the subsection, we provide a statistical application. More
precisely, we shall apply the method of proving Theorem
\ref{mdp-thm1} and \ref{mdp-thm2} to the least squares estimator
$\hat\theta_n^2$ and the Yule-Walker estimator $\tilde{\theta}_n$.

\begin{prop}\label{mdp-app}
Assume that the conditions (1), (2) and (3) are satisfied, then for
any $r>0$, we have
$$
\lim_{n\to\infty}\frac{1}{b_n^2}\log \pp\left(\frac{
\sqrt{n}}{b_n(1-\theta_n^2)^{1/2}}|\hat\theta_n-\theta_n|\ge
r\right)=-\frac{r^2}{2}
$$
and
$$
\lim_{n\to\infty}\frac{1}{b_n^2}\log \pp\left(\frac{
\sqrt{n}}{b_n(1-\theta_n^2)^{1/2}}|\tilde\theta_n-\theta_n|\ge
r\right)=-\frac{r^2}{2}.
$$
\end{prop}

\begin{rem} A recent paper by Giraitis and Phillips \cite{GiPh} (also see, Phillips and Magdalinos
\cite{Ph-Ma07}), established the asymptotic distribution of the
least square estimator $\hat\theta_n$ in a stationary first-order AR
model when $n(1-\theta_n)\to \infty$, i.e.,
$$
(1-\theta_n^2)^{-1/2}n^{1/2}(\hat\theta_n-\theta_n)\xrightarrow{d}
N(0,1).
$$
\end{rem}
\begin{rem}
For the case of $\theta_n\equiv\theta\in (-1,1)$, Djellout et al.
\cite{DGW06} derived the moderate deviations of $\hat\theta_n$ and
$\tilde\theta_n$ as a consequence of their general results on the
moderate deviation of moving average processes, but with an extra
and strong condition that the law of $\xi_0$ satisfies a log-Sobolev
inequality (though their method go far beyond the regression model).
In \cite{MS}, the authors gave the moderate deviations of
$\hat\theta_n$ and $\tilde\theta_n$, where they removed the
assumption of log-Sobolev inequality on the driven variable.
\end{rem}
\section{Some preparations and auxiliary results}

\subsection{Autoregressive representation for the covariance
process} For any $n$, by the stationarity
 of $X_{k,n}$ ($k=0,1,\cdots, n$), the distribution law of $X_{k+l,n}X_{k,n}$ is the same with
$X_{l,n}X_{0,n}$. For any $0\le l\le M$, let $C_{l,n}:=\ee
X_{k+l,n}X_{k,n}$ and it is easy to check that
 \begin{equation}\label{cl}
  C_{l,n}=\theta_n^l\ee X_{0,n}^2=\theta_n^l\sum_{k=0}^\infty\theta_n^{2k}\ee\xi_{0}^2
  =\ee C^{*}_{l,n}
 \end{equation}
where $C^{*}_{l,n}$ is defined in (\ref{C}). In addition, let
\begin{equation}\label{zu}
 Z_{k,l,n}=X_{k+l,n}X_{k,n}-C_{l,n},\ \ U_{k,l,n}=\theta_n X_{k+l-1,n}\xi_{k}+\theta_n\xi_{k+l}X_{k-1,n}
 +\xi_{k+l}\xi_k-\theta_n^l\ee\xi_0^2.
\end{equation}
We have the following autoregressive representation for the
covariance process.

 \begin{lem}\label{lem1} Under the above notions, for any $n> l$, we have
\begin{equation}\label{lem1-1}
 Z_{k,l,n}=\theta_n^2 Z_{k-1,l,n}+U_{k,l,n}, \ \ k=1,\cdots, n-l,
\end{equation}
and
 \begin{equation}\label{lem1-2}
  C_{l,n}^{*}-C_{l,n}=\frac{\bar
  U_{l,n}}{(1-\theta_n^2)}+\frac{\theta_n^2(Z_{0,l,n}-Z_{n-l,l,n})}{(n-l)(1-\theta_n^2)},
\end{equation}
where
$$
\bar
  U_{l,n}=\frac{1}{n-l}\sum_{k=1}^{n-l} U_{k,l,n}.
$$
 \end{lem}
\begin{proof}
 The proof of the lemma is easy, so omitted.
\end{proof}

\begin{lem}\label{mdp-lem1}
  Under the assumptions of Theorem \ref{mdp-thm1}, for any $0\le l\le M$ and $r>0$, we have
  $$
  \lim_{n\to\infty}\frac{1}{b_n^2}\log \pp\left(\frac{\sqrt{1-\theta_n^2}|Z_{0,l,n}-Z_{n-l,l,n}|}{b_n\sqrt{n-l}}\ge r\right)=-\infty.
  $$
   \end{lem}
\begin{proof} For every $n$, from the stationarity of $\{X_{k,n}\}_{0\le k\le n}$, we have
 $$
 \aligned
 & \pp\left(|Z_{0,l,n}-Z_{n-l,l,n}|\ge \frac{rb_n\sqrt{n-l}}{\sqrt{1-\theta_n^2}}\right)\\
= & \pp\left(|X_{l,n}X_{0,n}-X_{n,n}X_{n-l,n}|\ge \frac{rb_n\sqrt{n-l}}{\sqrt{1-\theta_n^2}}\right) \\
\le & 2\pp\left(|X_{0,n}X_{l,n}|\ge
\frac{rb_n\sqrt{n-l}}{2\sqrt{1-\theta_n^2}}\right)\\
 \le & 4\pp\left(X_{0,n}^2\ge
\frac{rb_n\sqrt{n-l}}{2\sqrt{1-\theta_n^2}}\right),
 \endaligned
 $$
where the last inequality follows from the well-known:
$$
2|X_{0,n}X_{l,n}|\le X_{0,n}^2+ X_{l,n}^2.
$$
 Now since
 $$
 X_{0,n}^2=\left(\sum_{p=0}^\infty\theta_n^p\xi_{-p}\right)^2
 \le \left(\sum_{p=0}^\infty\theta_n^p\right)\left(\sum_{p=0}^\infty\theta_n^p\xi_{-p}^2\right)
 =\frac{1}{1-\theta_n}\sum_{p=0}^\infty\theta_n^p\xi_{-p}^2
 $$
 and Markov's inequality, we have for $\lambda_n:=(1-\theta_n)^{2}\alpha$,
 $$
 \pp\left(X_{0,n}^2\ge \frac{rb_n\sqrt{n-l}}{2\sqrt{1-\theta_n^2}}\right)\le
 \exp\left(-\lambda_n r\frac{b_n\sqrt{n-l}}{2\sqrt{1-\theta_n^2}}\right)\ee e^{\lambda_n X_{0,n}^2}.
 $$
 But by Jensen's inequality,
 $$
  \ee e^{\lambda_n X_{0,n}^2}\le \ee\exp\left((1-\theta_n)\sum_{p=0}^\infty\theta_n^p\alpha\xi_{-p}^2\right)
  \le (1-\theta_n)\sum_{p=0}^\infty\theta_n^p\ee e^{\alpha\xi_{-p}^2}=\ee e^{\alpha\xi_{0}^2}.
 $$
 Summarizing the previous estimates we obtain
 $$
 \pp\left(|Z_{0,l,n}-Z_{n-l,l,n}|\ge \frac{rb_n\sqrt{n-l}}{\sqrt{1-\theta_n^2}}\right)\le
 4\exp\left(-\lambda_n r\frac{b_n\sqrt{n-l}}{2\sqrt{1-\theta_n^2}}\right)\ee e^{\alpha\xi_{0}^2},
 $$
 which yields the desired result by using the assumption
 $$
 b_n\to\infty,\ \ \ \frac{\sqrt n(1-\theta_n)^2}{b_n}\to\infty.
$$
\end{proof}

\subsection{Some properties of the sequence $\{U_{k,l,m,n}\}$.}

For all $n\ge 1$, $0\le l\le M$, $1\le k\le n-l$, $m>2M$, set
$$
X_{k-1,m,n}=\xi_{k-1}+\theta_n\xi_{k-2}+\cdots+\theta_n^{m-2}\xi_{k-m+1}=\sum_{j=0}^{m-2}\theta_n^j\xi_{k-1-j,}
$$
and
\begin{equation}\label{mdp-m-f}
\aligned
U_{k,l,m,n}=&\theta_nX_{k+l-1,m,n}\xi_k+\xi_{k+l}X_{k-1,m,n}\theta_n+\xi_{k+l}\xi_k-\theta_n^l\ee\xi_0^2\\
=&\sum_{j=1}^{m-1}\theta_n^j\xi_{k+l-j}\xi_k+\sum_{j=1}^{m-1}\xi_{k+l}\xi_{k-j}\theta_n^j+\xi_{k+l}\xi_k-\theta^l_n\ee\xi^2_0.
 \endaligned
\end{equation}

For any $n, l$, it is easy to see that $\{U_{k,l,m,n}\}_{1\le k\le
n-l}$ is a strictly stationary sequence with $m+l$-dependent
structure. Furthermore, the sequence $\{U_{k,l,m,n}\}_{1\le k\le
n-l}$ has the following properties.

\begin{prop}\label{prop-U}
\begin{enumerate}[{\rm i)}]
\item For any $0\le l\le M$, $1\le k\le n-l$,
  \begin{equation}\label{U-1}
\ee(U_{k,l,m,n})=0.
 \end{equation}

\item If $k\ne i$,
 \begin{equation}\label{U-2}
\ee(U_{k,0,m,n}U_{i,0,m,n})=0.
 \end{equation}
\item If $l\ne 0$, $k>i$,
 \begin{equation}\label{U-7}
 \aligned
 \ee(U_{k,l,m,n}U_{i,l,m,n})
 =
 \theta^{2l}_n(\ee\xi^2_0)^2\left(1_{A_1}+\sum_{q=0}^{m-1-2l}\theta_n^{2q}1_{A_2}\right)
\endaligned
 \end{equation}
 where the sets $A_1,A_2$ are defined by $A_1=\{i+l>k\}, A_2=\{i+l=k\}$.

 \item If $l\ne 0$,
\begin{equation}\label{U-8}
  \ee(U_{k,l,m,n}^2)=\left(\theta_n^{2l}\ee\xi_0^4+\left(1-2\theta_n^{2l}+2\sum_{j=1}^{m-1}\theta_n^{2j}\right)(\ee\xi_0^2)^2\right).
  \end{equation}

   \item
\begin{equation}
  \ee(U_{k,0,m,n}^2)=\ee\xi_0^4+(\ee\xi_0^2)^2\left[4\sum_{j=1}^{m-1}\theta_n^{2j}-1\right].
  \end{equation}
\end{enumerate}

\end{prop}

\begin{proof} Without loss of generality, we can assume that
$i<k$ and for any $k$, let
$$
\FF_{k}:=\sigma(\xi_{i}; -\infty<i\le k).
$$

{\bf Proof of i)}
 The claim
(\ref{U-1}) is easy to be obtained by the properties of conditional
expectation.

{\bf Proof of ii)}
 Since $\ee(U_{k,0,m,n}|\FF_{k-1})=0$, and
$U_{i,0,m,n}$ is measurable with respect to $\FF_{k-1}$, then we
have
$$
\aligned
 \ee(U_{k,0,m,n}U_{i,0,m,n})
 =&\ee[U_{i,0,m,n}\ee(U_{k,0,m,n}|\FF_{k-1})]=0.
\endaligned
$$

{\bf Proof of iii)} Let
$$
\Delta_{1,k,l}:=\sum_{j=1}^{m-1}\theta_n^j\xi_{k+l-j}\xi_k,\ \
\Delta_{2,k,l}:=\sum_{j=1}^{m-1}\xi_{k+l}\xi_{k-j}\theta_n^j.
$$
then it is easy to check that $U_{i,l,m,n}$ is measurable with
respect to $\FF_{k+l-1}$ and
$$
\ee(U_{k,l,m,n}|\FF_{k+l-1})=\Delta_{1,k,l}-\theta_n^l\ee\xi_0^2,
$$
So we have
$$
\ee(U_{k,l,m,n}U_{i,l,m,n})=\ee[U_{i,l,m,n}(\Delta_{1,k,l}-\theta_n^l\ee\xi_0^2)]=\ee(U_{i,l,m,n}\Delta_{1,k,l}).
$$
Next we need to calculate the following four terms:
$$
(1)\ \ee(\Delta_{1,k,l}\Delta_{1,i,l}),\ \ (2)\
\ee(\Delta_{1,k,l}\Delta_{2,i,l}),\ \ (3)\
\ee(\Delta_{1,k,l}\xi_{i+l}\xi_i),\ \ (4)\
\ee(\Delta_{1,k,l})\theta^l_n\ee\xi^2_0.
$$
First, we can observe that
\begin{itemize}
 \item when $i+l>k$, then there exists $1\le j,q\le m-1$
  such that \
  $$\ee(\xi_{k+l-j}\xi_k\xi_{i+l-q}\xi_i)\ne 0;$$
 \item when $i+l=k$, then there exists $1\le j,q\le m-1$
  such that \
  $$\ee(\xi_{k+l-j}\xi_k\xi_{i+l}\xi_{i-q})\ne 0;$$
   \item when $i+l=k$, then there exists $1\le j\le m-1$
  such that \
  $$\ee(\xi_{i+l}\xi_i\xi_{k+l-j}\xi_k)\ne 0.$$
\end{itemize}
Let $A_1=\{i+l>k\}, A_2=\{i+l=k\}$. Therefore, we have
$$
\ee(\Delta_{1,k,l}\Delta_{1,i,l})=\ee\left(\sum_{j=1}^{m-1}\theta_n^j\xi_{k+l-j}\xi_k\right)
\left(\sum_{q=1}^{m-1}\theta_n^q\xi_{i+l-q}\xi_i\right)=(\theta^l_n\ee\xi^2_0)^2(1+1_{A_1}),
$$
where we take $j=l=q$ and use the fact that under the case $i+l>k$,
we may choose $j=k+l-i, q=i+l-k$.
 Similarly, we have
$$
\ee(\Delta_{1,k,l}\Delta_{2,i,l})=\ee\left(\sum_{j=1}^{m-1}\theta_n^j\xi_{k+l-j}\xi_k\right)
\left(\sum_{q=1}^{m-1}\xi_{i+l}\xi_{i-q}\theta_n^q\right)=(\ee\xi_{0}^2)^21_{A_2}\sum_{q=1}^{m-1-2l}\theta_n^{2q+2l},
$$
$$
 \ee(\Delta_{1,k,l}\xi_{i+l}\xi_i)=\ee\left(\xi_{i+l}\xi_i\sum_{j=1}^{m-1}\theta_n^j\xi_{k+l-j}\xi_k\right)
 =(\theta^l_n\ee\xi^2_0)^21_{A_2}
$$
and
$$
\ee(\Delta_{1,k,l})\theta^l_n\ee\xi^2_0=(\theta^l_n\ee\xi^2_0)^2.
$$
From the above discussion and the definition of $U_{i,l,m,n}$, the
proof of iii) is completed.

{\bf Proofs of iv) and v)} Since
 $$
 \aligned
\ee U^2_{k,l,m,n}
=&\ee\left(\sum_{j=1}^{m-1}\theta_n^j\xi_{k+l-j}\xi_k+\sum_{j=1}^{m-1}\xi_{k+l}\xi_{k-j}\theta_n^j+\xi_{k+l}\xi_k-\theta^l_n\ee\xi^2_0\right)^2\\
=:&\ee(\Delta_1+\Delta_2+\Delta_3+\Delta_4)^2,
\endaligned
$$
then it is easy to see
$$
\aligned &
\ee(\Delta_1\Delta_3)=\ee(\Delta_2\Delta_3)=\ee(\Delta_2\Delta_4)=0,\\
& \ee\Delta_1^2=\begin{cases}
\theta_n^{2l}\ee\xi_0^4+(\ee\xi_0^2)^2\left(\sum_{j=1}^{m-1}\theta_n^{2j}-\theta_n^{2l}\right),\
\ & l\ne 0\\
 (\ee\xi_0^2)^2\sum_{j=1}^{m-1}\theta_n^{2j},\ \ & l=0,
\end{cases}\\
& \ee(\Delta_1\Delta_4)=\begin{cases} -\theta_n^{2l}(\ee\xi_0^2)^2,\
\ & l\ne 0\\
0,\ \ & l=0,
\end{cases}\\
& \ee(\Delta_3^2)=\begin{cases} (\ee\xi_0^2)^2,\
\ & l\ne 0\\
\ee\xi_0^4,\ \ & l=0,
\end{cases}\\
& \ee(\Delta_1\Delta_2)=\begin{cases} 0,\
\ & l\ne 0\\
(\ee\xi_0^2)^2\sum_{j=1}^{m-1}\theta_n^{2j},\ \ & l=0,
\end{cases}\\
& \ee(\Delta_3\Delta_4)=\begin{cases} 0,\
\ & l\ne 0\\
-(\ee\xi_0^2)^2,\ \ & l=0.
\end{cases}
\\
&
 \ee(\Delta_4^2)=\theta_n^{2l}(\ee\xi_0^2)^2,\
\ee(\Delta_2^2)=(\ee\xi_0^2)^2\sum_{j=1}^{m-1}\theta_n^{2j},\ \ \
\forall\ 0\le l\le M,
\endaligned$$
which yields the desired results.
\end{proof}

\begin{prop}\label{prop-UU} Let $1\le i<k$ and $0\le l,q\le M$.
\begin{enumerate}[{\rm (a)}]
\item If $l\ne 0$, then $\ee(U_{i,0,m,n}U_{k,l,m,n})=0$.

\item If $l\ne 0$, then $$\ee(U_{k,0,m,n}U_{i,l,m,n})=2\theta_n^l(\ee\xi_0^2)^2\left(1_{A_1}+1_{A_2}\sum_{j=0}^{m-1-l}\theta_n^{2j}\right)$$
where the events $A_1,A_2$ are defined in Proposition \ref{prop-U}.

\item If $0<l<q$, then
$$\ee(U_{i,l,m,n}U_{k,q,m,n})=\theta_n^{l+q}(\ee\xi_0^2)^2\left(1_{A_1}+1_{A_2}\sum_{j=0}^{m-1-(l+q)}\theta_n^{2j}\right).$$

\item If $0<q<l$, then
$$
\aligned
 \ee(U_{i,l,m,n}U_{k,q,m,n})
 =&\theta_n^{l+q}(\ee\xi_0^2)^2\left(1_{A_1}+1_{A_2}\sum_{j=0}^{m-1-(l+q)}\theta_n^{2j}\right)\\
 &+\theta_n^{l-q}(\ee\xi_0^2)^2\left(1_{E_1}+1_{E_2}\sum_{j=0}^{m-1-(l-q)}\theta_n^{2j}\right)
  \endaligned
  $$
where the events $E_1,E_2$ are defined by
$$
E_1=\{i+l>k+q\},\ \ \ E_2=\{i+l=k+q\}.
$$

\end{enumerate}
\end{prop}

\begin{proof} The proofs of the proposition are similar to the one of
Proposition \ref{prop-U}.

 {\bf Proof of (a)} Since $U_{i,0,m,n}$ is
measurable with respect to $\FF_{i}$, then we have for $i<k$,
$$
\aligned
&\ee(U_{i,0,m,n}U_{k,l,m,n})=\ee[U_{i,0,m,n}\ee(U_{k,l,m,n}|\FF_{i})]\\
=&\ee\left[U_{i,0,m,n}
\left(\sum_{j=1}^{m-1}\theta_n^j\xi_{k+l-j}\xi_k-\theta^l_n\ee\xi^2_0\right)\right]=0.
\endaligned
$$

{\bf Proof of (b)} Since $i<k$,
$$
\aligned
  \ee(U_{k,0,m,n}U_{i,l,m,n})=&
  \ee\left\{\left(2\sum_{p=1}^{m-1}\theta_n^p\xi_{k-p}\xi_k
  +\xi_{k}^2-\ee\xi^2_0\right)\right.\\
  &\left.\times\left(\sum_{j=1}^{m-1}\theta_n^j\xi_{i+l-j}\xi_i+\sum_{j=1}^{m-1}\xi_{i+l}\xi_{i-j}\theta_n^j
  +\xi_{i+l}\xi_i-\theta^l_n\ee\xi^2_0\right)\right\}\\
  =: &\ee(\Delta_{1}+\Delta_2+\Delta_3)(\Gamma_1+\Gamma_2+\Gamma_3+\Gamma_4),
\endaligned
$$
then it is easy to check that
$$
\ee\Delta_1\Gamma_4=\ee\Delta_2\Gamma_2=\ee\Delta_2\Gamma_3=\ee\Delta_3\Gamma_2=\ee\Delta_3\Gamma_3=0
$$
and
$$
\ee\Delta_2\Gamma_1=\ee\Delta_3\Gamma_4=\theta_n^l(\ee\xi_0^2)^2,\ \
\ee\Delta_2\Gamma_4=\ee\Delta_3\Gamma_1=-\theta_n^l(\ee\xi_0^2)^2.
$$
Furthermore, we have
\begin{itemize}
 \item  when $k<i+l$, then
$\ee\Delta_1\Gamma_1=2\theta_n^l(\ee\xi_0^2)^2$;

  \item when $k=i+l$, then
$\ee\Delta_1\Gamma_2=2\sum\limits_{j=1}^{m-1-l}\theta_n^{l+2j}(\ee\xi_0^2)^2$;

 \item when $k=i+l$, then
$\ee\Delta_1\Gamma_3=2\theta_n^l(\ee\xi_0^2)^2.$
\end{itemize}
So the desired result (b) is obtained.

{\bf Proof of (c)}  Since $i<k$ and $l<q$, then $U_{i,l,m,n}$ is
measurable with respect to $\FF_{k+q-1}$. Hence we have
$$
\aligned
& \ee(U_{i,l,m,n}U_{k,q,m,n})= \ee[U_{i,l,m,n}\ee(U_{k,q,m,n}|\FF_{k+q-1})]\\
 =& \ee\left[U_{i,l,m,n}
\left(\sum_{p=1}^{m-1}\theta_n^p\xi_{k+q-p}\xi_k-\theta^q_n\ee\xi^2_0\right)\right]\\
=&\ee\left[\left(\sum_{j=1}^{m-1}\theta_n^j\xi_{i+l-j}\xi_i+\sum_{j=1}^{m-1}\xi_{i+l}\xi_{i-j}\theta_n^j
  +\xi_{i+l}\xi_i-\theta^l_n\ee\xi^2_0\right)
\sum_{p=1}^{m-1}\theta_n^p\xi_{k+q-p}\xi_k\right].
\endaligned
$$
By the similar discussions as (b), we have
\begin{itemize}
\item
 $$ \aligned
\ee\left(\sum_{j=1}^{m-1}\theta_n^j\xi_{i+l-j}\xi_i\right)\left(\sum_{p=1}^{m-1}\theta_n^p\xi_{k+q-p}\xi_k\right)
=\theta_n^{l+q}(\ee\xi_0^2)^2(1+1_{A_1});
\endaligned
$$
\item when $i+l=k$, we have
$$
\ee\left(\sum_{j=1}^{m-1}\xi_{i+l}\xi_{i-j}\theta_n^j\right)\left(\sum_{p=1}^{m-1}\theta_n^p\xi_{k+q-p}\xi_k\right)
=\sum_{j=1}^{m-1-(l+q)}\theta_n^{l+q+2j}(\ee\xi_0^2)^2;
$$
\item when $i+l=k$, we have
$$
\ee\left(\xi_{i+l}\xi_i\sum_{p=1}^{m-1}\theta_n^p\xi_{k+q-p}\xi_k\right)
=\theta_n^{l+q}(\ee\xi_0^2)^2;
$$

\item for the last term,
$$
-\theta^l_n\ee\xi^2_0\ee\left(\sum_{p=1}^{m-1}\theta_n^p\xi_{k+q-p}\xi_k\right)
=-\theta_n^{l+q}(\ee\xi_0^2)^2.
$$
\end{itemize}
From the above discussion, the proof of (c) is completed.

{\bf Proof of (d)} Since $i<k$ and $q<l$, and
$$
\aligned
  \ee(U_{i,l,m,n}U_{k,q,m,n})=&
  \ee\left\{\left(\sum_{j=1}^{m-1}\theta_n^j\xi_{i+l-j}\xi_i+\sum_{j=1}^{m-1}\xi_{i+l}\xi_{i-j}\theta_n^j
  +\xi_{i+l}\xi_i-\theta^l_n\ee\xi^2_0\right)\right.\\
  &\left.\times\left(\sum_{p=1}^{m-1}\theta_n^p\xi_{k+q-p}\xi_k+\sum_{p=1}^{m-1}\xi_{k+q}\xi_{k-p}\theta_n^p
  +\xi_{k+q}\xi_k-\theta^q_n\ee\xi^2_0\right)\right\}\\
  =: &\ee(\hat\Delta_{1}+\hat\Delta_2+\hat\Delta_3+\hat\Delta_4)(\hat\Gamma_1+\hat\Gamma_2+\hat\Gamma_3+\hat\Gamma_4),
\endaligned
$$
then it is easy to see that
$$
\ee\hat\Delta_{1}\hat\Gamma_3=\ee\hat\Delta_{2}\hat\Gamma_3=\ee\hat\Delta_{3}\hat\Gamma_3=
\ee\hat\Delta_{4}\hat\Gamma_3=\ee
\hat\Delta_{2}\hat\Gamma_4=\ee\hat\Delta_{3}\hat\Gamma_4=\ee\hat\Delta_{4}\hat\Gamma_2=0
$$
and
$$
\ee\hat\Delta_{1}\hat\Gamma_4=\ee\hat\Delta_{4}\hat\Gamma_1=-\ee\hat\Delta_{4}\hat\Gamma_4=-\theta_n^{l+q}(\ee\xi_0^2)^2.
$$
In addition, we have
$$
\aligned
 \ee\hat\Delta_{1}\hat\Gamma_1=\theta_n^{l+q}(\ee\xi_0^2)^2(1+1_{A_1}),
\endaligned
$$
$$
\ee\hat\Delta_{2}\hat\Gamma_1=\theta_n^{l+q}(\ee\xi_0^2)^2\sum_{j=1}^{m-1-(l+q)}\theta_n^{2j}1_{A_2}
$$
and
$$
\ee\hat\Delta_{3}\hat\Gamma_1=\theta_n^{l+q}(\ee\xi_0^2)^21_{A_2}.
$$
Similarly, we can observe that
\begin{itemize}
 \item when $i+l>k+q$, then there exists $1\le j,q\le m-1$
  such that \
  $$\ee(\xi_{k+q}\xi_{k-p}\xi_{i+l-j}\xi_i)\ne 0;$$
 \item when $i+l=k+q$, then there exists $1\le j,q\le m-1$
  such that \
  $$\ee(\xi_{k+q}\xi_{k-p}\xi_{i+l}\xi_{i-j})\ne 0;$$
   \item when $i+l=k+q$, then there exists $1\le p\le m-1$
  such that \
  $$\ee(\xi_{i+l}\xi_i\xi_{k+q}\xi_{k-p})\ne 0.$$
\end{itemize}
Hence we have
$$
\ee\hat\Delta_{1}\hat\Gamma_2=\theta_n^{l-q}(\ee\xi_0^2)^21_{E_1},
$$
$$
\ee\hat\Delta_{2}\hat\Gamma_2=\theta_n^{l-q}(\ee\xi_0^2)^2\sum_{j=1}^{m-1-(l-q)}\theta_n^{2j}1_{E_2}
$$
and
$$
\ee\hat\Delta_{3}\hat\Gamma_2=\theta_n^{l-q}(\ee\xi_0^2)^21_{E_2}.
$$
Combining the above results, we complete the proof of (d).
\end{proof}

\begin{prop}\label{prop-UUU} When $i=k$, we have
 \begin{enumerate}[{\rm (1)}]
\item  If $0<l<q$, we have
$$\ee(U_{i,l,m,n}U_{i,q,m,n})=\theta_n^{l+q}(\ee\xi_0^4)-2\theta_n^{l+q}(\ee\xi_0^2)^2+\theta_n^{q-l}(\ee\xi_0^2)^2\sum_{j=0}^{m-1-(q-l)}\theta_n^{2j}.$$
\item If $0<q<l$, we have
$$\ee(U_{i,l,m,n}U_{i,q,m,n})=\theta_n^{l+q}(\ee\xi_0^4)-2\theta_n^{l+q}(\ee\xi_0^2)^2+\theta_n^{l-q}(\ee\xi_0^2)^2\sum_{p=0}^{m-1-(l-q)}\theta_n^{2p}.$$
\item If $l=q\ne 0$, we have
$$\ee(U_{i,l,m,n}^2)=\theta_n^{2l}(\ee\xi_0^4)-2\theta_n^{2l}(\ee\xi_0^2)^2+(\ee\xi_0^2)^2\left(2\sum_{j=1}^{m-1}\theta_n^{2j}+1\right).$$
\item If $q>0$, we have
$$\ee(U_{i,0,m,n}U_{i,q,m,n})=\theta_n^{q}(\ee\xi_0^4)-\theta_n^{q}(\ee\xi_0^2)^2+2(\ee\xi_0^2)^2\theta_n^q\sum_{j=1}^{m-1-q}
\theta_n^{2j}.$$
\end{enumerate}

\end{prop}
\begin{proof} From $$
\aligned
  \ee(U_{i,l,m,n}U_{i,q,m,n})=&\ee\left\{
  \left(\sum_{j=1}^{m-1}\theta_n^j\xi_{i+l-j}\xi_i+\sum_{j=1}^{m-1}\xi_{i+l}\xi_{i-j}\theta_n^j
  +\xi_{i+l}\xi_i-\theta^l_n\ee\xi^2_0\right)\right.\\
  &\left.\times\left(\sum_{p=1}^{m-1}\theta_n^p\xi_{i+q-p}\xi_i+\sum_{p=1}^{m-1}\xi_{i+q}\xi_{i-p}\theta_n^p
  +\xi_{i+q}\xi_i-\theta^q_n\ee\xi^2_0\right)\right\}\\
  =: &\ee(\tilde\Delta_{1}+\tilde\Delta_2+\tilde\Delta_3+\tilde\Delta_4)(\tilde\Gamma_1+\tilde\Gamma_2+\tilde\Gamma_3+\tilde\Gamma_4),
\endaligned
$$
we know that for any $0<l,q\le M$,
$$
\ee\tilde\Delta_{1}\tilde\Gamma_2=\ee\tilde\Delta_{2}\tilde\Gamma_1=\ee\tilde\Delta_{2}\tilde\Gamma_3
=\ee\tilde\Delta_{2}\tilde\Gamma_4=\ee\tilde\Delta_{3}\tilde\Gamma_2=\ee\tilde\Delta_{3}\tilde\Gamma_4=
\ee\tilde\Delta_{4}\tilde\Gamma_2=\ee\tilde\Delta_{4}\tilde\Gamma_3=0,
$$
$$
\ee\tilde\Delta_{1}\tilde\Gamma_4=\ee\tilde\Delta_{4}\tilde\Gamma_1=-\ee\tilde\Delta_{4}\tilde\Gamma_4=-\theta_n^{l+q}(\ee\xi_0^2)^2
$$
and
$$
\aligned
&\ee\tilde\Delta_{1}\tilde\Gamma_1=\begin{cases}\theta_n^{2l}\ee\xi_0^4+(\ee\xi_0^2)^2\left(\sum_{j=1}^{m-1}\theta_n^{2j}-\theta_n^{2l}\right),\
\ & l=q\\
\theta_n^{l+q}\ee\xi_0^4+(\ee\xi_0^2)^2\left(\sum_{p=1}^{m-1-(l-q)}\theta_n^{l-q+2p}-\theta_n^{l+q}\right),\
\ & l>q\\
\theta_n^{l+q}\ee\xi_0^4+(\ee\xi_0^2)^2\left(\sum_{p=1}^{m-1-(q-l)}\theta_n^{q-l+2p}-\theta_n^{l+q}\right),\
\ & l<q
\end{cases},\\
& \ee\tilde\Delta_{1}\tilde\Gamma_3=\begin{cases}
\theta_n^{l-q}(\ee\xi_0^2)^2,&  l>q\\
0,&  l\le q
\end{cases},\\
& \ee\tilde\Delta_{3}\tilde\Gamma_1=\begin{cases}
\theta_n^{q-l}(\ee\xi_0^2)^2,&  q>l\\
0,&  q\le l
\end{cases},\\
& \ee\tilde\Delta_{2}\tilde\Gamma_2=\begin{cases}
(\ee\xi_0^2)^2\sum_{j=1}^{m-1}\theta_n^{2j},&  l=q\\
0,&  l\ne q
\end{cases},\\
& \ee\tilde\Delta_{3}\tilde\Gamma_3=\begin{cases}
(\ee\xi_0^2)^2,&  q=l\\
0,&  q\ne l
\end{cases}.
\endaligned
$$
So the results (1)-(3) hold. Furthermore, (4) can be obtained by the
following observation
$$
\aligned
\ee(U_{i,0,m,n}U_{i,q,m,n})=\ee[U_{i,0,m,n}\ee(U_{i,q,m,n}|\FF_{i+q-1})]
=\ee\left(U_{i,0,m,n}\sum_{p=1}^{m-1}\theta_n^p\xi_{i+q-p}\xi_i\right).
\endaligned
$$
\end{proof}

\subsection{Moderate deviation for $m$-dependent sequence with
unbounded $m$}

Before giving our proofs of the main results, it is necessary to
give the following moderate deviation principle for $m$-dependent
random variables with unbounded $m$. For the readability of the
paper, we postpone its proof to Appendix.

\begin{lem}\label{lem-miao}
For each $n=1,2,\ldots,$ let $m=m(n)$ be specified and suppose that
$\{X_{1,n},\ldots,X_{n,n}\}$ be a sequence of strict stationary
$m$-dependent random variables with zero means. Moreover, we assume
the following conditions hold:

\begin{itemize}
\item[(A)] there exists a positive $0<\gamma<1$ such that the moderate deviation scale $(b_n)$ satisfies
$$
 b_n\to\infty,\ \ \ \  \frac{b_nm^{1+\gamma}}{\sqrt{n}}\to0;
$$
\item[(B)] for some $M>0$,
 $$
\frac{n}{b_n^2m}\int^\infty_Me^x\pp\left(|X_{1,n}|\ge\frac{\sqrt
nx}{ b_nm} \right)dx\to 0;
$$
\item[(C)] for any $\varepsilon>0$,
$$
\left(\frac{\sqrt{n}}{b_n}\right)^{2+\frac{2}{1+\gamma}}\pp\left(|X_{1,n}|>\varepsilon\left(\frac{\sqrt{n}}{b_n}\right)^{1-\frac{1}{1+\gamma}}\right)\to
0;
$$
\item[(D)] there exists a constant $0<\sigma^2<\infty$, such that
$$
 \lim_{n\to\infty}m^{-1}Var(X_{1,n}+\cdots+X_{m,n})=\sigma^2
 $$
 and
 $$
 \lim_{n\to\infty}m^{-1}\sum_{i=1}^mi\ee(X_{1,n}X_{i+1,n})=0.
$$
\end{itemize}
 Then for any $\lambda\in\rr$, we have
 $$
 \lim_{n\to\infty}\frac{1}{b_n^2}\log\ee\exp\left(\lambda\frac{b_n}{\sqrt n}\sum_{i=1}^nX_{i,n}\right)=\frac{\lambda^2\sigma^2}{2}.
 $$
Furthermore, by the G\"artner-Ellis theorem (see \cite{DeZe}), for
any $r>0$, we get
$$
\lim_{n\to\infty}\frac{1}{b_n^2}\log\pp\left(\frac{1}{b_n\sqrt
n}\left|\sum_{i=1}^nX_{i,n}\right|\ge
r\right)=-\frac{r^2}{2\sigma^2}.
$$
\end{lem}

\begin{rem} In \cite{Chen}, Chen established the
moderate deviation for $m$-dependent random vectors with fixed
parameter $m$. Recently, Miao and Yang \cite{MY} proved the
following moderate deviation, which extended Chen's result from
fixed $m$ to unbounded $m$ for $\rr$-valued $m$-dependent sequence:

{\it Assume that
  \begin{equation}\label{mdp-exp}
 \sup_{n}\ee\exp\{\alpha|X_{1,n}|\}<\infty,\ \ \ \text{for some}\ \
 \alpha>0
  \end{equation}
 and
\begin{equation}\label{mdp-exp1}
 b_n\to\infty,\ \ \ \  \frac{b_nm^{2}}{\sqrt{n}}\to0.
\end{equation}
In addition, if the condition (D) hold, then for any
$\lambda\in\rr$, we have
$$
 \lim_{n\to\infty}\frac{1}{b_n^2}\log\ee\exp\left(\lambda\frac{b_n}{\sqrt n}\sum_{i=1}^nX_{i,n}\right)=\frac{\lambda^2\sigma^2}{2}.
$$
} It is easy to see that the condition (\ref{mdp-exp}) and
(\ref{mdp-exp1}) imply the conditions (A), (B) and (C). But, the
condition (\ref{mdp-exp}) is not easy to check in the process of
proving our main results, so we need develop a new moderate
deviation for $m$-dependent sequence with unbounded $m$, that is,
Lemma \ref{lem-miao}.
\end{rem}

\section{Proof of Theorem \ref{mdp-thm1}}

\subsection{Asymptotic term and moderate deviations}

We have the following useful results, based on the properties of the
sequence $\{U_{k,l,m,n}\}_{1\le k\le n-l}$.

\begin{cor}\label{cor-0} Let $m:=m(n)$ denote the subsequence of $n$ such that $m(1-\theta_n)\to
\infty$ as $n\to\infty$. Then we have
\begin{equation}\label{cor-1}
\lim_{n\to\infty}\frac{1-\theta_n^2}{m}\sum_{k=1}^mk\ee(U_{1,l,m,n}U_{k+1,l,m,n})=0
\end{equation}
and
\begin{equation}\label{cor-2}
 \lim_{n\to\infty}\frac{1-\theta_n^2}{m}Var(U_{1,l,m,n}+\cdots+U_{m,l,m,n})=4(\ee\xi_0^2)^2.
  \end{equation}
\end{cor}

\begin{proof} {\bf  Case $l\ne 0$.} From iii) in Proposition \ref{prop-U}, it is easy to see
$$
\aligned
&\lim_{n\to\infty}\frac{1-\theta_n^2}{m}\sum_{k=1}^mk\ee(U_{1,l,m,n}U_{k+1,l,m,n})\\
=&\lim_{n\to\infty}\frac{1-\theta_n^2}{m}\sum_{k=1}^lk\ee(U_{1,l,m,n}U_{k+1,l,m,n})\\
=&\lim_{n\to\infty}\frac{1-\theta_n^2}{m}\left[\sum_{k=1}^{l-1}k\ee(U_{1,l,m,n}U_{k+1,l,m,n})
+l\ee(U_{1,l,m,n}U_{l+1,l,m,n})\right]\\
=&\lim_{n\to\infty}\frac{1-\theta_n^2}{m}\left[\sum_{k=1}^{l}k(\theta_n^l\ee\xi_0^2)^2
+l(\ee\xi_{0}^2)^2\sum_{q=1}^{m-1-2l}\theta_n^{2q+2l}\right]=0.
\endaligned
$$
Furthermore, since
$$
\aligned
 & Var(U_{1,l,m,n}+\cdots+U_{m,l,m,n})\\
 =& \sum_{k=1}^m\ee
 U^2_{k,l,m,n}+2\sum_{k=1}^{m-1}\sum_{q=k+1}^m\ee(U_{k,l,m,n}U_{q,l,m,n})\\
 =& \sum_{k=1}^m\ee
 U^2_{k,l,m,n}+2\sum_{k=1}^{m-l}\sum_{q=k+1}^{k+l}\ee(U_{k,l,m,n}U_{q,l,m,n}),
\endaligned
$$
then, by iii) and iv) in Proposition \ref{prop-U}, we have
$$
\aligned
\sum_{k=1}^{m-l}\sum_{q=k+1}^{k+l}\ee(U_{k,l,m,n}U_{q,l,m,n})
=&(m-l)\left(l(\theta_n^l\ee\xi_0^2)^2+(\ee\xi_{0}^2)^2\sum_{q=1}^{m-1-2l}\theta_n^{2q+2l}\right)
\endaligned
$$
and
$$
\sum_{k=1}^m\ee
 U^2_{k,l,m,n}=m\left(\theta_n^{2l}\ee\xi_0^4+\left(1-2\theta_n^{2l}+2\sum_{j=1}^{m-1}\theta_n^{2j}\right)(\ee\xi_0^2)^2\right).
$$
Hence it follows that
$$
\aligned
 &Var(U_{1,l,m,n}+\cdots+U_{m,l,m,n})\\
=&m\theta_n^{2l}\ee\xi_0^4+\left(m+[2(m-l)l-2m]\theta_n^{2l}\right)(\ee\xi_0^2)^2\\
&\ \ +
\left(2m\sum_{j=1}^{m-1}\theta_n^{2j}+2(m-l)\theta_n^{2l}\sum_{j=1}^{m-1-2l}\theta_n^{2j}\right)(\ee\xi_0^2)^2.
\endaligned
$$
By the assumption $(1-\theta_n)m\to \infty$ (which implies
$\theta_n^{m}\to 0$), we have
$$
\lim_{n\to\infty}\frac{1-\theta_n^2}{m}Var(U_{1,l,m,n}+\cdots+U_{m,l,m,n})=4(\ee\xi_0^2)^2.
$$

{\bf Case $l=0$.} From ii) in Proposition \ref{prop-U}, we have
$$
\lim_{n\to\infty}\frac{1-\theta_n^2}{m}\sum_{k=1}^mk\ee(U_{1,0,m,n}U_{k+1,0,m,n})=0
$$
and
$$
Var(U_{1,0,m,n}+\cdots+U_{m,0,m,n})
 = \sum_{k=1}^m\ee
 U^2_{k,0,m,n}.
$$
By v) in Proposition \ref{prop-U}, it follows that
\begin{equation}\label{U-00}
\ee
 U^2_{k,0,m,n}=\ee\xi_0^4+(\ee\xi_0^2)^2\left[4\sum_{j=1}^{m-1}\theta_n^{2j}-1\right]
\end{equation}
then we have
$$
\lim_{n\to\infty}\frac{1-\theta_n^2}{m}Var(U_{1,l,m,n}+\cdots+U_{m,l,m,n})=4(\ee\xi_0^2)^2.
$$
\end{proof}

Before giving the following proposition, we need to mention the {\bf
claim}: owing to the conditions
$$n(1-\theta_n)\to\infty\ \  \text{and}\ \
\frac{b_n}{(1-\theta_n)^2\sqrt n}\to 0,$$
 there must exist a subsequence
$m=m(n)$ such that
\begin{equation}\label{m-p}
m(1-\theta_n)\to \infty,\ \ \ \frac{m(1-\theta_n)}{|\log(1-\theta_n)|}\to \infty\ \ \text{and}\ \
\frac{b_nm^{5/3}}{\sqrt{n}}\to0.
\end{equation}
For instance, we can take
$$
m=(1-\theta_n)^{-6/5}.
$$

Now, based on the above preparations, we have the following result.
\begin{prop}\label{mdp-prop}
Under the assumptions of Theorem \ref{mdp-thm1}, for any
$\lambda\in\rr$, we have
$$
 \lim_{n\to\infty}\frac{1}{b_n^2}\log\ee\exp\left\{\lambda\frac{b_n\sqrt{1-\theta_n^2}}{\sqrt{n-l}}\sum_{k=1}^{n-l}U_{k,l,m,n}\right\}
 =2\lambda^2(\ee\xi_0^2)^2,\ \ \ \  0\le l\le M,
$$
where the sequence $\{m\}$ satisfies the properties in (\ref{m-p}).
Furthermore, for any $r>0$,
$$
 \lim_{n\to\infty}\frac{1}{b_n^2}\log\pp\left(\frac{\sqrt{1-\theta_n^2}}{b_n\sqrt{n-l}}\left|\sum_{k=1}^{n-l}U_{k,l,m,n}\right|\ge r\right)
 =-\frac{r^2}{8(\ee\xi_0^2)^2},\ \ \ \ 0\le l\le M.
$$
\end{prop}
\begin{proof}  Set
$$K_m(\theta_n)=\sum_{j=1}^{m-1}\theta_n^{j}
$$
then it is easy to see that
\begin{equation}\label{mdp-prop-0}
\lim_{n\to\infty}\theta_n^{2m}=0\ \ \ \text{and} \ \ \
\lim_{n\to\infty}(1-\theta_n^2) K_m(\theta_n)= 2.
\end{equation}
For every $0\le l\le M$, let
$$
X_{1,l,n}:=\sqrt{1-\theta_n^2}U_{1,l,m,n}
$$
then by the properties of $U_{1,l,m,n}$, we have
$$
\lim_{n\to\infty}\ee X_{1,0,n}^2=4(\ee\xi_0^2)^2,\ \ \text{and}\ \  \lim_{n\to\infty}\ee
X_{1,l,n}^2=2(\ee\xi_0^2)^2, \ \ l\ne 0.
$$
Furthermore, from Corollary \ref{cor-0}, for any $0\le l\le M$, it
follows that
 \begin{equation}\label{mdp-prop-1}
 \lim_{n\to\infty}\frac{1}{m}Var(X_{1,l,n}+\cdots+X_{m,l,n})=4(\ee\xi_0^2)^2
  \end{equation}
    and
 \begin{equation}\label{mdp-prop-2}
 \lim_{n\to\infty}\frac{1}{m}\sum_{k=1}^mk\ee(X_{1,l,n}X_{k+1,l,n})=0.
 \end{equation}
Next, we need to check the conditions (B) and (C) of Lemma
\ref{lem-miao} for the random variable $X_{1,l,n}$, namely, for any
$M>0$,
 \begin{equation}\label{C-C}
\frac{n}{b_n^2m}\int^\infty_Me^x\pp\left(|X_{1,l,n}|\ge\frac{\sqrt
nx}{ b_nm} \right)dx\to 0
\end{equation}
and for any $\varepsilon>0$,
\begin{equation}\label{C-D}
\left(\frac{\sqrt{n}}{b_n}\right)^{16/5}\pp\left(|X_{1,l,n}|>\varepsilon\left(\frac{\sqrt{n}}{b_n}\right)^{2/5}\right)\to
0,
\end{equation}
where we take $\gamma=2/3$.
However, from the definition of $X_{1,l,n}$,
$$
X_{1,l,n}=\sqrt{1-\theta_n^2}\left(\sum_{j=1}^{m-1}\theta_n^j\xi_{1+l-j}\xi_1+\sum_{j=1}^{m-1}\xi_{1+l}\xi_{1-j}\theta_n^j+\xi_{1+l}\xi_1-\theta^l_n\ee\xi^2_0\right)
$$
it is enough to show that (\ref{C-C}) and (\ref{C-D}) hold for the
term
$\sqrt{1-\theta_n^2}\sum_{j=1}^{m-1}\theta_n^j\xi_{1+l-j}\xi_1$, and
the proofs of the others are similar. By the conditions
$$
m(1-\theta_n^2)\to\infty \ \ \  \text{and}\ \  \ \frac{b_nm^{5/3}}{\sqrt
n}\to 0,
$$
we know that for all $n$ sufficient large
$$
\frac{\sqrt n}{ b_nm}\sqrt{1-\theta_n^2}\ge \frac{\sqrt n}{
b_nm^{3/2}}\ge\left(\frac{\sqrt n}{ b_n}\right)^{1/10}.
$$
From Cauchy-Schwarz's inequality, we have
$$
\aligned
 &\ee\exp\left\{\frac{\alpha(1-\theta_n^2)}{3}\sum_{j=1}^{m-1}\theta_n^j|\xi_{1+l-j}\xi_1|\right\}\\
 \le & \ee\exp\left\{\frac{\alpha(1-\theta_n^2)}{6}\sum_{j=1}^{m-1}\theta_n^{j}(\xi_{1+l-j}^2+\xi_1^2)\right\}\\
 \le& \left(\ee\exp\left\{\frac{\alpha(1-\theta_n^2)}{3}\sum_{j=1}^{m-1}\theta_n^{j}\xi_{1+l-j}^2\right\}\right)^{1/2}
 \left(\ee\exp\left\{\frac{\alpha(1-\theta_n^2)}{3}\sum_{j=1}^{m-1}\theta_n^{j}\xi_{1}^2\right\}\right)^{1/2}\\
 \le &\ee\exp\left\{\frac{\alpha(1-\theta_n^2)}{3}K_m(\theta_n)\xi_{1}^2\right\}\le \ee e^{\alpha \xi_1^2},
 \endaligned
$$
so, for any $M>0$,
 \begin{equation}
 \aligned
&\frac{n}{b_n^2m}\int^\infty_Me^x\pp\left(\left|(1-\theta_n^2)\sum_{j=1}^{m-1}\theta_n^j\xi_{1+l-j}\xi_1\right|\ge\frac{\sqrt
nx}{ b_nm}\sqrt{(1-\theta_n^2)} \right)dx\\
\le
&\frac{n}{b_n^2}\int^\infty_Me^x\pp\left(\frac{\alpha(1-\theta_n^2)}{3}\sum_{j=1}^{m-1}\theta_n^j\left|\xi_{1+l-j}\xi_1\right|\ge\left(\frac{\sqrt
n}{ b_n}\right)^{1/10}\frac{\alpha x}{3} \right)dx\\
\le & \ee e^{\alpha
\xi_1^2}\frac{n}{b_n^2}\int^\infty_M\exp\left\{-\left[\left(\frac{\sqrt
n}{ b_n}\right)^{1/10}\frac{\alpha }{3}-1 \right]x\right\}dx\to 0.
\endaligned
\end{equation}
In addition, from the fact that
$$
\frac{(1-\theta_n^2)^{1/2}}{(\sqrt{n}/b_n)^{-3/10}}\to \infty,
$$
we have,
for any $\varepsilon>0$,
\begin{equation}
\aligned
 &\left(\frac{\sqrt{n}}{b_n}\right)^{16/5}\pp\left(\left|(1-\theta_n^2)\sum_{j=1}^{m-1}\theta_n^j\xi_{1+l-j}\xi_1\right|>\varepsilon\left(\frac{\sqrt{n}}{b_n}\right)^{2/5}\sqrt{(1-\theta_n^2)}\right)\\
\le &\left(\frac{\sqrt{n}}{b_n}\right)^{16/5}\pp\left(\left|(1-\theta_n^2)\sum_{j=1}^{m-1}\theta_n^j\xi_{1+l-j}\xi_1\right|>\varepsilon\left(\frac{\sqrt{n}}{b_n}\right)^{1/10}\right)\\
\le & \ee e^{\alpha
\xi_1^2}\left(\frac{\sqrt{n}}{b_n}\right)^{16/5}\exp\left\{-\varepsilon\left(\frac{\sqrt{n}}{b_n}\right)^{1/10}\frac{\alpha
}{3}\right\}\to 0.
\endaligned
\end{equation}
Therefore, the conditions in Lemma \ref{lem-miao} are satisfied and
the desired results of the proposition can be obtained.
\end{proof}

\subsection{Exponential approximation} In this subsection, we shall establish the asymptotic negligibility of the term
$\frac{\sqrt{1-\theta_n^2}}{b_n\sqrt{n-l}}\sum_{k=1}^{n-l}(U_{k,l,m,n}-U_{k,l,n})$
as $n\to\infty$.
For all $p\ge 0$ and $k\ge 1$, set
\begin{equation}\label{W}
W_{k,p}=\xi_k\xi_{k-p}.
\end{equation}

\begin{lem}\label{mdp-lem2}
Let the assumptions of Proposition \ref{mdp-prop} hold.

(1) There exist $\alpha_0$ and $\beta_0$ such that for all $p\ge 1$, $n\ge 1$ and $t\ge 0$
 \begin{equation}\label{mdp-lem2-1}
  \pp\left(\max_{j\le n}\left|\sum_{k=1}^jW_{k,p}\right|\ge t\right)\le 36\exp\left(-\frac{t^2}{\alpha_0n+\beta_0t}\right).
 \end{equation}

 (2) For all $t>0$, there exist $N\ge 1$, $A,B>0$ such that, for all $n\ge N$ and $0\le l\le M$,
 \begin{equation}\label{mdp-lem2-2}
 \aligned
 & \pp\left(\max_{j\le n-l}\left|\sum_{k=1}^{j}(U_{k,l,m,n}-U_{k,l,n})\right|\ge tb_n\frac{\sqrt{n-l}}{\sqrt{1-\theta_n^2}}\right)\\
  \le&
72\left(1-\exp\left(-\frac{b_n^2t^2}{(At+B)K_n\theta_n^m\sqrt{1-\theta_n^2}}\right)
\right)^{-1}\exp\left(-\frac{b_n^2t^2}{(At+B)K_n\theta_n^m\sqrt{1-\theta_n^2}}\right),
 \endaligned
\end{equation}
where $K_n=(1-\theta_n)^{-2}$ and $K_n\theta_n^m\sqrt{1-\theta_n^2}\to0$ as $n\to\infty$.

(3) For all $t>0$,
$$\lim_{n\to\infty}\frac{1}{b_n^2}\log\pp\left(\frac{\sqrt{1-\theta_n^2}}{b_n\sqrt{n-l}}\left|\sum_{k=1}^{n-l}(U_{k,l,m,n}-U_{k,l,n})\right|\ge t\right)= -\infty.
$$
\end{lem}

\begin{proof} (1) This is Lemma 17 in \cite{MaMl}.

(2) Firstly, we have
 \begin{equation}\label{mdp-lem2-5}
 \aligned
  \left|\sum_{k=1}^{j}(U_{k,l,m,n}-U_{k,l,n})\right|
  \le& \left|\sum_{k=1}^{j}\theta_n\xi_k(X_{k+l-1,n}-X_{k+l-1,m,n})\right|\\
  &\ \ +\left|\sum_{k=1}^{j}\theta_n\xi_{k+l}(X_{k-1,n}-X_{k-1,m,n})\right|.
  \endaligned
 \end{equation}
 Since
 $$
 X_{k+l-1,n}-X_{k+l-1,m,n}=\theta_n^{m-1}\sum_{p=0}^\infty\theta_n^p\xi_{k+l-m-p},
 $$
 we can get
 $$
 \aligned
  \left|\sum_{k=1}^{j}\theta_n\xi_k(X_{k+l-1,n}-X_{k+l-1,m,n})\right|\le
  \theta_n^m\sum_{p=0}^\infty\theta_n^p\left|\sum_{k=1}^{j}W_{k,m+p-l}\right|.
 \endaligned
 $$
 Now it is not difficult to show the fact: for any $n\ge 1$,
 $$
  K_n=\sum_{p=0}^\infty(p+1)\theta_n^p=(1-\theta_n)^{-2}.
 $$
 Hence, by (\ref{mdp-lem2-1}), we have
 \begin{equation}\label{mdp-lem2-7}
 \aligned
 &\pp\left(\max_{1\le j\le n-l} \left|\sum_{k=1}^{j}\theta_n\xi_k(X_{k+l-1,n}-X_{k+l-1,m,n})\right|
 >tb_n\frac{\sqrt{n-l}}{2\sqrt{1-\theta_n^2}}\right)\\
 \le & \pp\left(\sum_{p=0}^\infty(p+1)\frac{\theta_n^p}{p+1}\max_{1\le j\le n-l} \left|\sum_{k=1}^{j}W_{k,m+p-l}\right|  > \sum_{p=0}^\infty(p+1)\theta_n^p\frac{tb_n\sqrt{n-l}}{2K_n\sqrt{1-\theta_n^2}\theta_n^m}\right)\\
 \le & \sum_{p=0}^\infty\pp\left(\max_{1\le j\le n-l} \left|\sum_{k=1}^{j}W_{k,m+p-l}\right| >
 \frac{tb_n(p+1)\sqrt{n-l}}{2K_n\sqrt{1-\theta_n^2}\theta_n^m}\right)\\
 \le & 36\sum_{p=0}^\infty\exp\left(-\frac{b_n^2t^2_{m,p}(t)}{\alpha_0+\beta_0t_{m,p}(t)b_n/\sqrt{n-l}}\right)
 \endaligned
 \end{equation}
 where
 $$
 t_{m,p}(t)=\frac{t(p+1)}{2K_n\theta_n^m\sqrt{1-\theta_n^2}}.
 $$
 By noting that
 $$
 \frac{m(1-\theta_n)}{|\log(1-\theta_n)|}\to \infty\ \ \ \Longrightarrow \ \ \lim_{n\to\infty}K_n\theta_n^m\sqrt{1-\theta_n^2}=0
 $$
and from the assumption of $b_n$, there exist constants $N\in\mathbb{N}$, $A,B>0$, such that for all $n\ge N$, $l\ge 0$, and we obtain
  \begin{equation}\label{mdp-lem2-6}
   \frac{t^2_{m,p}(t)}{\alpha_0+\beta_0t_{m,p}(t)b_n/\sqrt{n-l}}\ge c(t)\frac{p+1}{K_n\theta_n^m\sqrt{1-\theta_n^2}},\ \ \ \
   c(t):=\frac{t^2}{At+B}.
  \end{equation}
Hence, by (\ref{mdp-lem2-7}) and (\ref{mdp-lem2-6}), we get
\begin{equation}
 \aligned
 &\pp\left(\max_{1\le j\le n-l} \left|\sum_{k=1}^{j}\theta_n\xi_k(X_{k+l-1,n}-X_{k+l-1,m,n})\right|
 >tb_n\frac{\sqrt{n-l}}{2\sqrt{1-\theta_n^2}}\right)\\
 \le & 36\sum_{p=0}^\infty\exp\left(-b_n^2c(t)\frac{p+1}{K_n\theta_n^m\sqrt{1-\theta_n^2}}\right)\\
= & 36\left(1-\exp\left(-\frac{b_n^2c(t)}{K_n\theta_n^m\sqrt{1-\theta_n^2}}\right)\right)^{-1}\exp\left(-\frac{b_n^2c(t)}{K_n\theta_n^m\sqrt{1-\theta_n^2}}\right).
 \endaligned
 \end{equation}
 For the same reason, we can give the estimate of the second term in (\ref{mdp-lem2-5}), so the proof of (\ref{mdp-lem2-2}) can be completed..

 (3) It follows obviously by (\ref{mdp-lem2-2}).
\end{proof}

At last, the proof of Theorem \ref{mdp-thm1} can be completed by Lemma \ref{lem1},
Proposition \ref{mdp-prop} and (3) in Lemma \ref{mdp-lem2}.

\section{Proof of Theorem \ref{mdp-thm2}}
Let $Y_{k,l,n}=X_{k+l,n}X_{k,n}-\ee
 X_{k+l,n}X_{k,n}$ for any $1\le k\le n$ and $0\le l\le M$.
 Since
\begin{equation}\label{mdp-thm2-p1}
\aligned
 \sum_{l=0}^Ma_{l,n}\sum_{k=1}^{n-l}Y_{k,l,n}
 =\sum_{k=1}^{n-M}\sum_{l=0}^Ma_{l,n}Y_{k,l,n}+\sum_{l=0}^{M-1}\sum_{k=n-M+1}^{n-l}a_{l,n}Y_{k,l,n},
\endaligned
\end{equation}
then the desired result is equivalent to showing
\begin{equation}\label{mdp-thm2-p2}
\lim_{n\to\infty}\frac{1}{b_n^2}\log\pp\left(\frac{(1-\theta_n^2)^{3/2}}{b_n\sqrt{n}}
\left|\sum_{k=1}^{n-M}\sum_{l=0}^Ma_{l,n}Y_{k,l,n}\right|\ge
r\right)
 =-\frac{r^2}{2\Sigma^2}
\end{equation}
and
\begin{equation}\label{mdp-thm2-p3}
\lim_{n\to\infty}\frac{1}{b_n^2}\log\pp\left(\frac{(1-\theta_n^2)^{3/2}}{b_n\sqrt{n}}
\left|\sum_{l=0}^{M-1}\sum_{k=n-M+1}^{n-l}a_{l,n}Y_{k,l,n}\right|\ge
r\right)
 =-\infty.
\end{equation}
As the similar proof of Theorem \ref{mdp-thm1}, in order to obtain
(\ref{mdp-thm2-p2}), it is enough to show that for any
$\lambda\in\rr$, it follows
\begin{equation}\label{mdp-thm2-p4}
 \lim_{n\to\infty}\frac{1}{b_n^2}\log\ee\exp\left\{\lambda\frac{b_n\sqrt{1-\theta_n^2}}{\sqrt{n}}
 \sum_{k=1}^{n-M}\sum_{l=0}^Ma_{l,n}U_{k,l,m,n}\right\}
 =\frac{\lambda^2\Sigma^2}{2},
\end{equation}
where the sequence $\{m\}$ satisfies the properties in (\ref{m-p}).
Let
\begin{equation}\label{Y-hat}
\hat Y_{k,m,n}:=\sum_{l=0}^Ma_{l,n}U_{k,l,m,n},
\end{equation}
 then it is
easy to see that $\{\hat Y_{k,m,n}\}_{1\le k\le n-M}$ is a strictly
stationary sequence with $m+M$-dependent structure. Hence from Lemma
\ref{lem-miao}, (\ref{mdp-thm2-p4}) is equivalent to proving
\begin{equation}\label{mdp-thm2-p5}
\lim_{n\to\infty}\frac{1-\theta_n^2}{m}\sum_{k=1}^mk\ee(\hat
Y_{1,m,n}\hat Y_{k+1,m,n})=0
\end{equation}
and
\begin{equation}\label{mdp-thm2-p6}
 \lim_{n\to\infty}\frac{1-\theta_n^2}{m}Var(\hat Y_{1,m,n}+\cdots+\hat Y_{m,m,n})=\Sigma^2.
  \end{equation}
 For $i<k$, we have
 $$
\aligned
 \hat Y_{i,m,n}\hat
Y_{k,m,n}=& a_{0,n}^2U_{i,0,m,n}U_{k,0,m,n}+\left(\sum_{l=1}^Ma_{l,n}U_{i,l,m,n}\right)\left(\sum_{q=1}^Ma_{q,n}U_{k,q,m,n}\right)\\
&\
+a_{0,n}U_{k,0,m,n}\sum_{l=1}^Ma_{l,n}U_{i,l,m,n}+a_{0,n}U_{i,0,m,n}\sum_{q=1}^Ma_{q,n}U_{k,q,m,n}.
\endaligned
$$
From (\ref{U-2}) and Proposition \ref{prop-UU}, it follows that
$$
\ee U_{i,0,m,n}U_{k,0,m,n}=0,  \ \ \ \ \ee
\left(\sum_{q=1}^MU_{i,0,m,n}U_{k,q,m,n}\right)=0,
$$
and
$$
\ee\left(U_{k,0,m,n}\sum_{l=1}^MU_{i,l,m,n}\right)
=2(\ee\xi_0^2)^2\sum_{l=1}^M\theta_n^l\left(1_{A_1}+1_{A_2}\sum_{j=0}^{m-1-l}\theta_n^{2j}\right).
$$
Since $i<k$ and
$$
\aligned
&\left(\sum_{l=1}^MU_{i,l,m,n}\right)\left(\sum_{q=1}^MU_{k,q,m,n}\right)\\
=&\left(\sum_{l=1}^{M-1}\sum_{q=l+1}^M+\sum_{q=1}^{M-1}\sum_{l=q+1}^M\right)U_{i,l,m,n}U_{k,q,m,n}
+\sum_{l=1}^{M}U_{i,l,m,n}U_{k,l,m,n},
\endaligned
$$
then by (\ref{U-7}) and Proposition \ref{prop-UU}, we have
$$
\sum_{l=1}^{M}\ee(U_{i,l,m,n}U_{k,l,m,n})
 =
 \sum_{l=1}^{M}\theta^{2l}_n(\ee\xi^2_0)^2\left(1_{A_1}+\sum_{q=0}^{m-1-2l}\theta_n^{2q}1_{A_2}\right),
$$
$$
\aligned
&\sum_{l=1}^{M-1}\sum_{q=l+1}^M\ee(U_{i,l,m,n}U_{k,q,m,n})\\
=&\sum_{l=1}^{M-1}\sum_{q=l+1}^M\left(\theta_n^{l+q}(\ee\xi_0^2)^2\left(1_{A_1}+1_{A_2}\sum_{j=0}^{m-1-(l+q)}\theta_n^{2j}\right)\right)
  \endaligned$$
and
$$
\aligned
 &\sum_{q=1}^{M-1}\sum_{l=q+1}^M\ee(U_{i,l,m,n}U_{k,q,m,n})\\
 =&\sum_{q=1}^{M-1}\sum_{l=q+1}^M\theta_n^{l+q}(\ee\xi_0^2)^2\left(1_{A_1}+1_{A_2}\sum_{j=0}^{m-1-(l+q)}\theta_n^{2j}\right)\\
 &+\sum_{q=1}^{M-1}\sum_{l=q+1}^M\theta_n^{l-q}(\ee\xi_0^2)^2\left(1_{E_1}+1_{E_2}\sum_{j=0}^{m-1-(l-q)}\theta_n^{2j}\right).
  \endaligned
  $$
Hence we can obtain
\begin{equation}\label{Y-2}
\aligned
 &\ee(\hat Y_{i,m,n}\hat
 Y_{k,m,n})\\
 =&2(\ee\xi_0^2)^2a_{0,n}\sum_{l=1}^Ma_{l,n}\theta_n^l\left(1_{A_1}+1_{A_2}\sum_{j=0}^{m-1-l}\theta_n^{2j}\right)\\
 &\
 +(\ee\xi^2_0)^2\sum_{l=1}^{M}a_{l,n}^2\theta^{2l}_n\left(1_{A_1}+1_{A_2}\sum_{q=0}^{m-1-2l}\theta_n^{2q}\right)\\
 &\ +
 (\ee\xi_0^2)^2\sum_{l=1}^{M-1}\sum_{q=l+1}^Ma_{l,n}a_{q,n}\theta_n^{l+q}\left(1_{A_1}+1_{A_2}\sum_{j=0}^{m-1-(l+q)}\theta_n^{2j}\right)\\
 &\ +(\ee\xi_0^2)^2\sum_{q=1}^{M-1}\sum_{l=q+1}^Ma_{l,n}a_{q,n}\theta_n^{l+q}\left(1_{A_1}+1_{A_2}\sum_{j=0}^{m-1-(l+q)}\theta_n^{2j}\right)\\
 &\ +(\ee\xi_0^2)^2\sum_{q=1}^{M-1}\sum_{l=q+1}^Ma_{l,n}a_{q,n}\theta_n^{l-q}\left(1_{E_1}+1_{E_2}\sum_{j=0}^{m-1-(l-q)}\theta_n^{2j}\right)\\
 =:& I_{1,i,k,n}+I_{2,i,k,n}+I_{3,i,k,n}+I_{4,i,k,n}+I_{5,i,k,n}.
\endaligned
\end{equation}
Furthermore, since
$$
\aligned
 \hat
Y_{i,m,n}^2=&a_{0,n}^2U_{i,0,m,n}^2
+2a_{0,n}U_{i,0,m,n}\sum_{l=1}^Ma_{l,n}U_{i,l,m,n}
+\left(\sum_{l=1}^Ma_{l,n}U_{i,l,m,n}\right)^2
\endaligned
$$
and
$$
\aligned \left(\sum_{l=1}^Ma_{l,n}U_{i,l,m,n}\right)^2=&
\left(\sum_{l=1}^{M-1}\sum_{q=l+1}^M+\sum_{q=1}^{M-1}\sum_{l=q+1}^M\right)a_{l,n}a_{q,n}U_{i,l,m,n}U_{i,q,m,n}\\
&\ \ \  \ \ \  +\sum_{l=1}^{M}a_{l,n}^2U_{i,l,m,n}^2,
\endaligned
$$
then from Proposition \ref{prop-U} and Proposition \ref{prop-UUU},
we have
\begin{equation}\label{Y-3}
\aligned
&\ee(\hat Y_{i,m,n}^2)\\
=& a_{0,n}^2\left(\ee\xi_0^4+(\ee\xi_0^2)^2\left[4\sum_{j=1}^{m-1}\theta_n^{2j}-1\right]\right)\\
&\
+2a_{0,n}\sum_{l=1}^Ma_{l,n}\left(\theta_n^{l}(\ee\xi_0^4)-\theta_n^{l}(\ee\xi_0^2)^2+2(\ee\xi_0^2)^2\theta_n^l\sum_{j=1}^{m-1-l}
\theta_n^{2j}\right)\\
&\ +
\sum_{l=1}^{M-1}\sum_{q=l+1}^Ma_{l,n}a_{q,n}\left(\theta_n^{l+q}(\ee\xi_0^4)-2\theta_n^{l+q}(\ee\xi_0^2)^2+\theta_n^{q-l}(\ee\xi_0^2)^2\sum_{j=0}^{m-1-(q-l)}\theta_n^{2j}\right)\\
&\
+\sum_{q=1}^{M-1}\sum_{l=q+1}^Ma_{l,n}a_{q,n}\left(\theta_n^{l+q}(\ee\xi_0^4)-2\theta_n^{l+q}(\ee\xi_0^2)^2+\theta_n^{l-q}(\ee\xi_0^2)^2\sum_{p=0}^{m-1-(l-q)}\theta_n^{2p}\right)\\
&\ +
\sum_{l=1}^{M}a_{l,n}^2\left(\theta_n^{2l}(\ee\xi_0^4)-2\theta_n^{2l}(\ee\xi_0^2)^2+(\ee\xi_0^2)^2\left(2\sum_{j=1}^{m-1}\theta_n^{2j}+1\right)\right).
\endaligned
\end{equation}

Now we prove the relations (\ref{mdp-thm2-p5}) and
(\ref{mdp-thm2-p6}). From (\ref{Y-2}), in order to show
(\ref{mdp-thm2-p5}), we only prove the following claim
\begin{equation}\label{Y-1}
\aligned
&\lim_{n\to\infty}\frac{1-\theta_n^2}{m}\sum_{k=1}^mkI_{1,1,k+1,n}\\
=&
2(\ee\xi_0^2)^2\lim_{n\to\infty}\frac{1-\theta_n^2}{m}a_{0,n}\sum_{k=1}^mk\sum_{l=1}^Ma_{l,n}\theta_n^l\left(1_{A_1}+1_{A_2}\sum_{j=0}^{m-1-l}\theta_n^{2j}\right)=0,
\endaligned
\end{equation}
and the proofs of other terms are similar. In fact, by the
definitions of $A_1,A_2$, we have
$$
\aligned
&\sum_{k=1}^mk\sum_{l=1}^Ma_{l,n}\theta_n^l\left(1_{A_1}+1_{A_2}\sum_{j=0}^{m-1-l}\theta_n^{2j}\right)\\
=&\sum_{l=1}^Ma_{l,n}\theta_n^l\sum_{k=1}^{M+1}k1_{A_1}
+\sum_{l=1}^Ma_{l,n}\theta_n^{l}l\left(\sum_{j=0}^{m-1-l}\theta_n^{2j}\right)
\endaligned
$$
which implies (\ref{Y-1}). Next we prove (\ref{mdp-thm2-p6}). Since
$$
Var(\hat Y_{1,m,n}+\cdots+\hat Y_{m,m,n})=\sum_{k=1}^m\ee\hat
Y_{k,m,n}^2+2\sum_{i=1}^{m-1}\sum_{k=i+1}^{m}\ee\hat Y_{k,m,n}\hat
Y_{i,m,n},
$$
then from (\ref{Y-3}), we have
$$
\aligned
 &\lim_{n\to\infty}\frac{1-\theta_n^2}{m} \sum_{k=1}^m\ee\hat
Y_{k,m,n}^2\\
=&\left(4a_{0}^2+4a_0\sum_{l=1}^M
a_l+2\sum_{l=1}^{M-1}\sum_{q=l+1}^M a_la_q+2\sum_{l=1}^M
a_l^2\right)(\ee\xi_0^2)^2.
\endaligned
$$
Moreover, by (\ref{Y-2}), we have
$$
\aligned
&\frac{1-\theta_n^2}{m}\sum_{i=1}^{m-1}\sum_{k=i+1}^{m}I_{1,i,k,n}\\
=&2(\ee\xi_0^2)^2a_{0,n}\frac{1-\theta_n^2}{m}\sum_{l=1}^Ma_{l,n}\theta_n^l\sum_{i=1}^{m-1}\left((l+1)+\sum_{j=0}^{m-1-l}\theta_n^{2j}\right)\\
\to &2(\ee\xi_0^2)^2a_{0}\sum_{l=1}^Ma_{l}.
\endaligned
$$
Similarly, we have
$$
\frac{1-\theta_n^2}{m}\sum_{i=1}^{m-1}\sum_{k=i+1}^{m}I_{2,i,k,n}
\to(\ee\xi_0^2)^2\sum_{l=1}^Ma_{l}^2,$$
$$
\frac{1-\theta_n^2}{m}\sum_{i=1}^{m-1}\sum_{k=i+1}^{m}I_{3,i,k,n}\to
(\ee\xi_0^2)^2\sum_{l=1}^{M-1}\sum_{q=l+1}^Ma_{l}a_{q},$$
$$
\frac{1-\theta_n^2}{m}\sum_{i=1}^{m-1}\sum_{k=i+1}^{m}I_{4,i,k,n}\to
(\ee\xi_0^2)^2\sum_{q=1}^{M-1}\sum_{l=q+1}^Ma_{l}a_{q},$$
$$
\frac{1-\theta_n^2}{m}\sum_{i=1}^{m-1}\sum_{k=i+1}^{m}I_{5,i,k,n}\to
(\ee\xi_0^2)^2\sum_{q=1}^{M-1}\sum_{l=q+1}^Ma_{l}a_{q},$$
 so, it
follows that
$$
\aligned
&\lim_{n\to\infty}\frac{1-\theta_n^2}{m}\sum_{i=1}^{m-1}\sum_{k=i+1}^{m}\ee\hat
Y_{k,m,n}\hat Y_{i,m,n}\\
=&
(\ee\xi_0^2)^2\left(2a_{0}\sum_{l=1}^Ma_{l}+\sum_{q=1}^{M-1}\sum_{l=q+1}^Ma_{l}a_{q}+\left(\sum_{q=1}^{M}a_q\right)^2\right).
\endaligned
$$ From the above discussion, we have
$$
\lim_{n\to\infty}\frac{1-\theta_n^2}{m}Var(\hat
Y_{1,m,n}+\cdots+\hat
Y_{m,m,n})=4\left(\sum_{j=0}^Ma_j\right)^2(\ee\xi_0^2)^2.
$$
At last, we need to show (\ref{mdp-thm2-p3}). Since for any $r>0$,
we have
$$
\aligned
 &\pp\left(\frac{(1-\theta_n^2)^{3/2}}{b_n\sqrt{n}}
\left|\sum_{l=0}^{M-1}\sum_{k=n-M+1}^{n-l}a_{l,n}Y_{k,l,n}\right|\ge
r\right)\\
\le
&\sum_{l=0}^{M-1}\sum_{k=n-M+1}^{n-l}\pp\left(\frac{(1-\theta_n^2)^{3/2}}{b_n\sqrt{n}}
\left|a_{l,n}Y_{k,l,n}\right|\ge \frac{2r}{M(M+1)}\right)
\endaligned
$$
then from the stationarity of $Y_{k,l,n} (k=0,1,\cdots, n-l)$ and
the fact that for any $l$, $|a_{l,n}|<N_l$ for some $N_l>0$, it is
enough to show that for any $0\le l\le M$,
\begin{equation}\label{Y-6}
\lim_{n\to\infty}\frac{1}{b_n^2}\log\pp\left(\frac{(1-\theta_n^2)^{3/2}}{b_n\sqrt{n}}
\left|Y_{k,l,n}\right|\ge \frac{2r}{N_lM(M+1)}\right)\to -\infty.
\end{equation}
However, from the definition of $Y_{k,l,n}$ and the fact that
$\ee(X_{l,n}X_{0,n})=\theta_n^l(1-\theta_n^2)^{-1}\ee\xi_0^2$, we have
$$
\aligned
 &\lim_{n\to\infty}\frac{1}{b_n^2}\log\pp\left(\frac{(1-\theta_n^2)^{3/2}}{b_n\sqrt{n}}
\left|Y_{k,l,n}\right|\ge \frac{2r}{N_lM(M+1)}\right)\\
\le
&\lim_{n\to\infty}\frac{1}{b_n^2}\log\pp\left(\frac{(1-\theta_n^2)^{3/2}}{b_n\sqrt{n}}
\left|X_{l,n}X_{0,n}-\ee(X_{l,n}X_{0,n})\right|\ge \frac{2r}{N_lM(M+1)}\right)\\
\le &
\lim_{n\to\infty}\frac{1}{b_n^2}\log\pp\left(\frac{(1-\theta_n^2)^{3/2}}{b_n\sqrt{n}}
\left|X_{l,n}X_{0,n}\right|\ge \frac{r}{N_lM(M+1)}\right)\\
\le &
\lim_{n\to\infty}\frac{1}{b_n^2}\log\pp\left(\frac{(1-\theta_n^2)^{1/2}}{b_n\sqrt{n}}
\left|X_{l,n}X_{0,n}\right|\ge \frac{r}{N_lM(M+1)}\right)\to
-\infty.
\endaligned
$$
Here the last limit is due to the similar proof in Lemma
\ref{mdp-lem1} .

\section{Proof of Proposition \ref{mdp-app}}

The proof of Proposition \ref{mdp-app} stems from the method of
Theorem \ref{mdp-thm2}. From the definition of $\hat\theta_n$, we
have
$$
\aligned & \frac{
\sqrt{n}}{b_n(1-\theta_n^2)^{1/2}}(\hat\theta_n-\theta_n)\\
 = & \frac{
\sqrt{n}}{b_n(1-\theta_n^2)^{1/2}}\frac{\sum_{k=1}^n(X_{k,n}X_{k-1,n}-\theta_nX_{k-1,n}^2)}{\sum_{k=1}^nX_{k-1,n}^2}\\
=&
\frac{\frac{(1-\theta_n^2)^{3/2}}{\sqrt{n}b_n}\sum_{k=1}^n(1-\theta_n^2)^{-1}(\ee\xi_0^2)^{-1}(X_{k,n}X_{k-1,n}-\theta_nX_{k-1,n}^2)}{(1-\theta_n^2)(\ee\xi_0^2)^{-1}\frac{1}{n}\sum_{k=1}^nX_{k-1,n}^2}=:\frac{r_n}{R_n}.
\endaligned
$$
Let us first prove that $(R_n-1)$ is negligible with respect to the
moderate deviation principle, i.e., to show that for any $r>0$,
\begin{equation}\label{mdp-app-1}
\lim_{n\to\infty}\frac{1}{b_n^2}\log\pp\left((1-\theta_n^2)(\ee\xi_0^2)^{-1}\frac{1}{n}\left|\sum_{k=1}^n(X_{k-1,n}^2-\ee
X_{k-1,n}^2)\right|>r\right)=-\infty,
\end{equation}
where we use the fact $\ee X_{k-1,n}^2=\ee
X_{0,n}^2=(1-\theta_n^2)^{-1}\ee\xi_0^2$. Since
$$
\aligned
&\pp\left((1-\theta_n^2)(\ee\xi_0^2)^{-1}\frac{1}{n}\left|\sum_{k=1}^n(X_{k-1,n}^2-\ee
X_{k-1,n}^2)\right|>r\right)\\
=&\pp\left(\frac{(1-\theta_n^2)^{3/2}}{b_n\sqrt
n}\left|\sum_{k=1}^n(X_{k-1,n}^2-\ee
X_{k-1,n}^2)\right|>\frac{r\sqrt{n}(\ee\xi_0^2)(1-\theta_n^2)^{1/2}}{b_n}\right),
\endaligned
$$
then by the condition (3) (which implies
$\sqrt{n(1-\theta_n^2)}b_n^{-1}\to\infty$) and Theorem
\ref{mdp-thm1} yields (\ref{mdp-app-1}). Next we only need to prove
that $r_n$ satisfies the moderate deviation principle. Let
$$
a_{0,n}=-(1-\theta_n^2)^{-1}(\ee\xi_0^2)^{-1}\theta_n\ \ \text{and}\
\  a_{1,n}=(1-\theta_n^2)^{-1}(\ee\xi_0^2)^{-1},
$$
then
$$
r_n=\frac{(1-\theta_n^2)^{3/2}}{\sqrt{n}b_n}\sum_{l=0}^1\sum_{k=0}^{n-1}a_{l,n}(X_{k+l,n}X_{k,n}-\ee
X_{k+l,n}X_{k,n}).
$$
Since $a_{0,n}\to -\infty, a_{1,n}\to\infty$, then we can not use
directly Theorem \ref{mdp-thm2} to prove the moderate deviation of
$r_n$. So we need to slightly modify the proof of Theorem
\ref{mdp-thm2}.

Now rewrite $r_n$ as
$$
\aligned
r_n=&\frac{(1-\theta_n^2)^{1/2}}{\sqrt{n}b_n(\ee\xi_0^2)}\sum_{k=1}^n\xi_k\sum_{p=0}^\infty\theta_n^p\xi_{k-1-p}\\
=: &\frac{1}{\sqrt{n}b_n}\sum_{k=1}^n\hat X_{k,n}.
\endaligned
$$
Let $m$ be a increasing sequence satisfying the properties in
(\ref{m-p}) and put
$$
\hat X_{k,m,n}=\frac{(1-\theta_n^2)^{1/2}}{(\ee\xi_0^2)}\sum_{p=0}^{m-1}\theta_n^p\xi_{k-1-p}\xi_k
$$
then $\{\hat X_{k,m,n}\}$ is a strictly stationary sequence with
$m$-dependent structure.
\begin{lem}\label{lem6-1}
For any $r>0$, we have
$$
\lim_{n\to\infty}\frac{1}{b_n^2}\log
\pp\left(\frac{1}{b_n\sqrt{n}}\left|\sum_{k=1}^n \hat X_{k,m,n}\right|\ge
r\right)=-\frac{r^2}{2}.
$$
\end{lem}

\begin{proof}
In order to obtain the desired result, it is enough to check the
conditions in Lemma \ref{lem-miao}. Firstly, it is easy to see that
$$
\lim_{n\to\infty}\ee\hat X_{k,m,n}^2=1
$$
and for any $k\ne j$,
$$
\ee (\hat X_{k,m,n}\hat X_{j,m,n})=0.
$$
Hence we have
$$
\lim_{n\to \infty}\frac{1}{m}Var(\hat X_{1,m,n}+\cdots+\hat X_{m,m,n})=1
$$
and
$$
\lim_{n\to\infty}\frac{1}{m}\sum_{i=1}^m
i\ee(\hat X_{1,m,n}\hat X_{i+1,m,n})=0.
$$
Moreover, by using the similar proofs of (\ref{C-C}) and
(\ref{C-D}), the conditions (B) and (C) in Lemma \ref{lem-miao}
hold. So we complete the proof by using Lemma \ref{lem-miao}.
\end{proof}
\begin{lem}\label{lem6-2} For any $r>0$, we have
$$
\lim_{n\to\infty}\frac{1}{b_n^2}\log
\pp\left(\frac{1}{b_n\sqrt{n}}\left|\sum_{k=1}^n (\hat X_{k,m,n}-\hat X_{k,n})\right|\ge
r\right)=-\infty.
$$
\end{lem}
\begin{proof} Since for any $0\le j\le n$,
$$
\aligned
\sum_{k=1}^j (\hat X_{k,m,n}-\hat X_{k,n})=&\frac{(1-\theta_n^2)^{1/2}}{(\ee\xi_0^2)}\sum_{k=1}^j \sum_{p=m}^\infty\theta_n^p\xi_{k-1-p}\xi_k\\
=&\frac{(1-\theta_n^2)^{1/2}}{(\ee\xi_0^2)}\theta_n^m\sum_{k=1}^j \sum_{p=0}^\infty\theta_n^p\xi_{k-1-p-m}\xi_k\\
=& \frac{(1-\theta_n^2)^{1/2}}{(\ee\xi_0^2)}\theta_n^m\sum_{p=0}^\infty\theta_n^p\sum_{k=1}^j W_{k,m+p+1}
\endaligned
$$
where $W_{k,m+p+1}$ is defined in (\ref{W}), then by the same proof
of Lemma \ref{mdp-lem2}, the desired result can be obtained.
\end{proof}
At last, Proposition \ref{mdp-app} can be given by using Lemma \ref{lem6-1} and \ref{lem6-2}.

\section{Appendix}

\begin{proof}[{\bf Proof of Lemma \ref{lem-miao}}]
For each $n$, let
$$
Y_{j,n}=\sum_{i=1}^mX_{(j-1)m+i,n}, \ \  \ 1\le j\le l
$$
 where $l:=l(n)=\max\{j: jm\le n\}$, then $\{Y_{1,n},\ldots,
 Y_{l,n}\}$ are $1$-dependent random variables. Furthermore, take
 $p=p(n)$, such that
\begin{equation}\label{p}
p(n)\to\infty\ \  \text{and}\ \
\frac{b_n(mp)^{1+\gamma}}{\sqrt{n}}\to 0, \ \ \text{as}\ \
n\to\infty
\end{equation}
and define
$$
 Z_{h,n}=\sum_{(h-1)p<j<hp}Y_{j,n},\ \ 1\le h\le t
$$
where $t:=t(n)=\max\{h,hp<l\}$, then $\{Z_{1,n},\ldots,
 Z_{t,n}\}$ is an i.i.d. random sequence, and we have the
 following relations
\begin{equation}\label{lem-miao-1}
\aligned
 \sum_{i=1}^nX_{i,n}=&\sum_{j=1}^lY_{j,n}+\sum_{i=lm+1}^nX_{i,n}\\
 =&
 \sum_{h=1}^tZ_{h,n}+\sum_{j=tp+1}^lY_{j,n}+\sum_{h=1}^tY_{hp,n}+\sum_{i=lm+1}^nX_{i,n}.
 \endaligned
\end{equation}

\begin{lem}\label{lem1-miao} Under the assumptions of Lemma \ref{lem-miao}, for any $\lambda\in\rr$, we have
$$
 \lim_{n\to\infty}\frac{1}{b_n^2}\log\ee\exp\left(\lambda\frac{b_n}{\sqrt
 n}\sum_{h=1}^tZ_{h,n}\right)=\frac{\lambda^2\sigma^2}{2},
$$
i.e., for any $r>0$
$$
 \lim_{n\to\infty}\frac{1}{b_n^2}\log\pp\left(\frac{1}{b_n\sqrt
 n}\left|\sum_{h=1}^tZ_{h,n}\right|>r\right)=-\frac{r^2}{2\sigma^2}.
$$
\end{lem}

\begin{proof} For $\tau>0$, define
$$
 X_{i,n}^\tau:=X_{i,n}\mathbb{I}_{\left\{|X_{i,n}|\le \tau \frac{\sqrt
 n}{b_n}\right\}},\ \ \ 1\le i\le n,
$$
$$
Y_{j,n}^{\tau}=\sum_{i=1}^mX_{(j-1)m+i,n}^\tau, \ \  \ 1\le j\le l
$$
and
$$
 Z_{h,n}^\tau=\sum_{(h-1)p<j<hp}Y_{j,n}^\tau,\ \ 1\le h\le t
$$
 where $l, p, t$ are defined in the above notations. Now we divide the proof
 into the following two steps.

{\bf Step 1.} We claim that for any $r>0$
\begin{equation}\label{Step-1}
 \lim_{n\to\infty}\frac{1}{b_n^2}\log\pp\left(\frac{1}{b_n\sqrt
 n}\left|\sum_{h=1}^tZ^\tau_{h,n}\right|>r\right)=-\frac{r^2}{2\sigma^2}.
\end{equation}

Since $\{Z_{1,n}^{\tau},\ldots,
 Z_{t,n}^{\tau}\}$ is an i.i.d. random sequence, then for
 $\lambda\in\rr$,
 $$
\aligned
  &\frac{1}{b_n^2}\log\ee\exp\left(\lambda\frac{b_n}{\sqrt
n}\sum_{h=1}^tZ_{h,n}^\tau\right)
=\frac{t}{b_n^2}\log\ee\exp\left(\lambda\frac{b_n}{\sqrt
n}Z_{1,n}^\tau\right)\\
=&\frac{t}{b_n^2}\log\left(1+\lambda\frac{b_n}{\sqrt n}\ee
Z_{1,n}^\tau+\frac{\lambda^2b^2_n}{
2n}\ee(Z_{1,n}^\tau)^2+O\left(\frac{\lambda^3b^3_n}{n^{3/2}}\ee(Z_{1,n}^\tau)^3\right)\right).
  \endaligned$$
First from the conditions (A), (B), (C) and Fubini Theorem, we have
\begin{equation}\label{Step-1-1}
\aligned
  &m \ee\left( X_{1,n}^2\mathbb{I}_{\left\{|X_{1,n}|> \tau
\frac{\sqrt
 n}{b_n}\right\}}\right)\\
 =&\frac{n}{b_n^2 m}\ee\left( \left(\frac{b_n m}{\sqrt n}X_{1,n}\right)^2\mathbb{I}_{\left\{|X_{1,n}|> \tau
\frac{\sqrt
 n}{b_n}\right\}}\right)
 \\
 =&\frac{2n}{b_n^2 m}\int_{0}^\infty x\pp\left(|X_{1,n}|\ge\frac{\sqrt
nx}{ b_nm}, |X_{1,n}|> \tau \frac{\sqrt
 n}{b_n} \right)dx\\
 =& \frac{2n}{b_n^2 m}\int_{\tau m}^\infty x\pp\left(|X_{1,n}|\ge\frac{\sqrt
nx}{ b_nm} \right)dx+\frac{2n}{b_n^2 m}\int_0^{\tau m}
x\pp\left(|X_{1,n}|>\tau \frac{\sqrt
 n}{b_n} \right)dx\\
 = &\frac{2n}{b_n^2 m}\int_{\tau m}^\infty x\pp\left(|X_{1,n}|\ge\frac{\sqrt
nx}{ b_nm} \right)dx+\frac{n\tau^2m}{b_n^2} \pp\left(|X_{1,n}|>\tau
\frac{\sqrt
 n}{b_n} \right)\\
 \le &\frac{2n}{b_n^2 m}\int_{\tau m}^\infty x\pp\left(|X_{1,n}|\ge\frac{\sqrt
nx}{ b_nm}
\right)dx+\tau^2\left(\frac{\sqrt{n}}{b_n}\right)^{2+\frac{1}{1+\gamma}}
\pp\left(|X_{1,n}|>\tau \frac{\sqrt n}{b_n} \right)\\
\to & 0,
 \endaligned
\end{equation}
where we utilized the fact: for all sufficiently large $n$,
$$
m\le \left(\frac{\sqrt{n}}{b_n}\right)^{\frac{1}{1+\gamma}}.
$$
Noting that $\ee X_{1,n}=0$ and the fact $(tmp)/n\to 1$, we have, by
(\ref{Step-1-1}),
$$
\aligned
 \frac{t}{b_n^2}\frac{\lambda b_n}{\sqrt n}|\ee
Z_{1,n}^\tau|\le &\frac{\lambda t(p-1)m}{b_n\sqrt n}\ee
\left(\left|X_{1,n}\right|\mathbb{I}_{\left\{|X_{1,n}|> \tau
\frac{\sqrt
 n}{b_n}\right\}}\right)\\
 \le &\frac{\lambda tpm}{\tau n}\ee\left(
X_{1,n}^2\mathbb{I}_{\left\{|X_{1,n}|> \tau \frac{\sqrt
 n}{b_n}\right\}}\right)\to 0.
 \endaligned
$$
From the definition of $Z_{1,n}^\tau$, we have
$$
\aligned
 \ee(Z_{1,n}^\tau)^2=&(p-1)\ee(Y_{1,n}^\tau)^2+2\sum_{j=1}^{p-2}\sum_{i=j+1}^{p-1}\ee(Y_{j,n}^\tau
 Y_{i,n}^\tau)\\
 =&(p-1)\ee(Y_{1,n}^\tau)^2+2\sum_{j=1}^{p-2}\ee(Y_{j,n}^\tau
 Y_{j+1,n}^\tau)+2\sum_{j=1}^{p-3}\sum_{i=j+2}^{p-1}\ee(Y_{j,n}^\tau)\ee(
 Y_{i,n}^\tau).
\endaligned
$$
Since the estimate (\ref{Step-1-1}) implies
$$
\aligned
 &\frac{1}{m}\ee(Y_{1,n}^\tau-Y_{1,n})^2\\
 =& \frac{1}{m}\ee\left(
X_{1,n}\mathbb{I}_{\left\{|X_{1,n}|> \tau \frac{\sqrt
 n}{b_n}\right\}}+\cdots+
X_{m,n}\mathbb{I}_{\left\{|X_{m,n}|> \tau \frac{\sqrt
 n}{b_n}\right\}}\right)^2\\
 \le & m \ee\left(
X_{1,n}^2\mathbb{I}_{\left\{|X_{1,n}|> \tau \frac{\sqrt
 n}{b_n}\right\}}\right)\to 0
 \endaligned
$$
then we have
$$
\frac1m\ee(Y_{1,n}^\tau)^2\to \sigma^2
$$
where we used the condition (D) and the triangle inequality in the
$L^2$ spaces:
$$
\|Y_{1,n}^\tau\|_2\le \|Y_{1,n}^\tau-Y_{1,n}\|_2+\|Y_{1,n}\|_2,\ \
\|Y_{1,n}\|_2\le\|Y_{1,n}^\tau-Y_{1,n}\|_2+\|Y_{1,n}^\tau\|_2.
$$
So we have
$$
 \frac{\lambda^2t }{n}(p-1)\ee(Y_{1,n}^\tau)^2=\frac{\lambda^2
tm(p-1)}{n}\frac{\ee(Y_{1,n}^\tau)^2}{m}\to \lambda^2\sigma^2.
$$
Similarly, we have
\begin{equation}
\aligned
 &\frac{1}{m}\sum_{i=1}^{m}i\ee|(X_{1,n}^\tau-X_{1,n})(X_{i+1,n}^\tau-X_{i+1,n})|\\
 \le & \frac{m+1}{2}\ee\left(
X_{1,n}^2\mathbb{I}_{\left\{|X_{1,n}|> \tau \frac{\sqrt
 n}{b_n}\right\}}\right)\to 0,
\endaligned
\end{equation}
\begin{equation}
\aligned
&\frac{1}{m}\sum_{i=1}^{m}i\ee|(X_{1,n}^\tau-X_{1,n})X_{i+1,n}|
\\
=& \frac{1}{m}\sum_{i=1}^{m}i\left\{\ee\left(
|X_{1,n}|\mathbb{I}_{\left\{|X_{1,n}|> \tau \frac{\sqrt
 n}{b_n}\right\}}
|X_{i+1,n}|\mathbb{I}_{\left\{|X_{i+1,n}|\le \tau \frac{\sqrt
 n}{b_n}\right\}}\right)\right.\\
& \ \ \ \left. + \ee\left( |X_{1,n}|\mathbb{I}_{\left\{|X_{1,n}|>
\tau \frac{\sqrt
 n}{b_n}\right\}}
|X_{i+1,n}|\mathbb{I}_{\left\{|X_{i+1,n}|> \tau \frac{\sqrt
 n}{b_n}\right\}}\right)\right\}\\
 \le & (m+1)\ee\left(
X_{1,n}^2\mathbb{I}_{\left\{|X_{1,n}|> \tau \frac{\sqrt
 n}{b_n}\right\}}\right)\to 0,
\endaligned
\end{equation}
and
\begin{equation}
\aligned
\frac{1}{m}\sum_{i=1}^{m}i\ee|(X_{i+1,n}^\tau-X_{i+1,n})X_{1,n}|\to
0,
\endaligned
\end{equation}
from which we can deduce
$$ \aligned
 &\frac{1}{m}\sum_{i=1}^{m}i\ee(X_{1,n}^\tau
 X_{i+1,n}^\tau)\\
 =&\frac{1}{m}\sum_{i=1}^{m}i \left[\ee(X_{1,n}
 X_{i+1,n})+\ee(X_{1,n}^\tau-X_{1,n})(X_{i+1,n}^\tau-X_{i+1,n})\right.\\
 &\ \ \ \  \left. +\ee(X_{1,n}^\tau-X_{1,n})X_{i+1,n}+\ee
 X_{1,n}(X_{i+1,n}^\tau-X_{i+1,n})\right]\to 0.
\endaligned
$$
 Hence we can get
$$
\aligned
 &\frac{t\lambda^2
}{n}\sum_{j=1}^{p-2}\ee(Y_{j,n}^\tau
 Y_{j+1,n}^\tau)\\
 = & \frac{\lambda^2 t(p-2)m}{n}\left\{\frac{1}{m}\sum_{i=1}^{m}i\ee(X_{1,n}^\tau
 X_{i+1,n}^\tau)+\frac{1}{m}\frac{m(m-1)}{2}(\ee X_{1,n}^\tau)^2\right\}\to 0
\endaligned
$$
and
$$
\frac{t\lambda^2
}{n}\sum_{j=1}^{p-3}\sum_{i=j+2}^{p-1}|\ee(Y_{j,n}^\tau)\ee(
 Y_{i,n}^\tau)|\le \frac{\lambda^2 tp^2m^2
}{n}(\ee X_{1,n}^\tau)^2\to 0
$$
where we used the condition (\ref{p}) and the fact that
$$
(\ee X_{1,n}^\tau)^2=\left(\ee X_{1,n}\mathbb{I}_{\left\{|X_{1,n}|>
\tau \frac{\sqrt
 n}{b_n}\right\}}\right)^2\le
\frac{b_n^2}{n\tau^2}\left(\ee
X_{1,n}^2\mathbb{I}_{\left\{|X_{1,n}|> \tau \frac{\sqrt
 n}{b_n}\right\}}\right)^2.
$$
So we have
\begin{equation}\label{lem-miao-3}
\frac{\lambda^2 t}{n}\ee(Z_{1,n}^\tau)^2\to \lambda^2\sigma^2.
\end{equation}
Furthermore, for any $\varepsilon>0$, we have for all $n$ sufficient
large,
$$
\aligned
 \frac{\lambda^3b_n t}{n^{3/2}}
\ee|Z_{1,n}^\tau|^3\le & \frac{\lambda^3b_n t}{n^{3/2}}
\ee\left(|Z_{1,n}^\tau|^2\left(\sum_{i=1}^{(p-1)m}|X_{i,n}^\tau|
\mathbb{I}_{\left\{|X_{i,n}^\tau|\le \varepsilon \frac{\sqrt
 n}{b_nmp}\right\}}\right)\right)\\
& \ \ \ \ +\frac{\lambda^3b_n t}{n^{3/2}}
\ee\left(|Z_{1,n}^\tau|^2\left(\sum_{i=1}^{(p-1)m}|X_{i,n}^\tau|
\mathbb{I}_{\left\{|X_{i,n}^\tau|> \varepsilon \frac{\sqrt
 n}{b_nmp}\right\}}\right)\right)\\
 \le &\frac{\varepsilon\lambda^3t}{n}\ee|Z_{1,n}^\tau|^2
 +\frac{\tau^3\lambda^3 tm^3p^3}{b_n^{2}}\pp\left\{|X_{1,n}|> \varepsilon \frac{\sqrt
 n}{b_nmp}\right\}\\
= &\frac{\varepsilon\lambda^3t}{n}\ee|Z_{1,n}^\tau|^2
 +\frac{\tau^3\lambda^3 tmp}{n}\frac{nm^2p^2}{b_n^{2}}\pp\left\{|X_{1,n}|> \varepsilon \frac{\sqrt
 n}{b_nmp}\right\}\\
 \le & \frac{\varepsilon\lambda^3t}{n}\ee|Z_{1,n}^\tau|^2
 +\frac{\tau^3\lambda^3 tmp}{n}\left(\frac{\sqrt n}{b_n}\right)^{2+\frac{2}{1+\gamma}}\pp\left\{|X_{1,n}|> \varepsilon \left(\frac{\sqrt
 n}{b_n}\right)^{1-\frac{1}{1+\gamma}}\right\},
\endaligned
$$
where we used the condition
$$
\lim_{n\to\infty}\frac{b_n(mp)^{1+\gamma}}{\sqrt n}=0\ \
\Longrightarrow \ \ \lim_{n\to\infty}\frac{mp}{(\sqrt
n/b_n)^{1/(1+\gamma)}}=0.
$$
Therefore, by noting $mpt/n\to1$ and the arbitrariness of
$\varepsilon$, we have, from the relation (\ref{lem-miao-3}) and the
condition (C),
$$
\frac{\lambda^3b_n t}{n^{3/2}} \ee|Z_{1,n}^\tau|^3\to 0.
$$
 Hence from the above discussions, we have
$$
\aligned
  &\frac{1}{b_n^2}\log\ee\exp\left(\lambda\frac{b_n}{\sqrt
n}\sum_{h=1}^tZ_{h,n}^\tau\right)\to \frac 12 \lambda^2\sigma^2,
  \endaligned$$
which yields, by the G\"artner-Ellis theorem, for any $r>0$
$$
 \lim_{n\to\infty}\frac{1}{b_n^2}\log\pp\left(\frac{1}{b_n\sqrt
 n}\left|\sum_{h=1}^tZ_{h,n}^\tau\right|>r\right)=-\frac{r^2}{2\sigma^2}.
$$

{\bf Step 2.} We shall prove that for any $r>0$ and $\tau>0$,
\begin{equation}\label{lem-miao-2}
\lim_{n\to \infty}\frac{1}{b_n^2}\log \pp\left(\frac{1}{b_n\sqrt
n}\left|\sum_{h=1}^t(Z_{h,n}^\tau-Z_{h,n})\right|\ge
r\right)=-\infty.
\end{equation}
Let $[a]$ denote the integral part of $a$, then for any $\lambda>0$,
\begin{equation}\label{Step-2-1}
\aligned
 &\frac{1}{b_n^2}\log \pp\left(\frac{1}{b_n\sqrt
n}\left|\sum_{h=1}^t(Z_{h,n}^\tau-Z_{h,n})\right|\ge r\right)\\
\le & -r\lambda+\frac{1}{b_n^2}\log \ee\exp\left(\frac{\lambda
b_n}{\sqrt
n}\sum_{h=1}^t\left|Z_{h,n}^\tau-Z_{h,n}\right|\right)\\
\le &  -r\lambda+\frac{t}{b_n^2}\log \ee\exp\left(\frac{\lambda
b_n}{\sqrt
n}\sum_{j=1}^{p-1}\left|Y_{j,n}^\tau-Y_{j,n}\right|\right)\\
\le & -r\lambda+\frac{t}{b_n^2}\log
\left(\ee\exp\left(2\frac{\lambda b_n}{\sqrt
n}\sum_{j=1}^{[(p-1)/2]}\left|Y_{2j,n}^\tau-Y_{2j,n}\right|\right)\right)^{1/2}\\
&\ \ \ \  \times \left(\ee\exp\left(2\frac{\lambda b_n}{\sqrt
n}\sum_{j=1}^{L}\left|Y_{2j-1,n}^\tau-Y_{2j-1,n}\right|\right)\right)^{1/2}\\
= & -r\lambda+\frac{t(p-1)}{2b_n^2}\log \ee\exp\left(2\frac{\lambda
b_n}{\sqrt n}\left|Y_{1,n}^\tau-Y_{1,n}\right|\right)
\endaligned
\end{equation}
where
$$
L:=L(n):=\begin{cases} \left[(p-1)/2\right]\ \ & \ \text{if $p-1$ is
even }\\
\left[(p-1)/2\right]+1\ \ & \ \text{if $p-1$ is odd. }
\end{cases}
$$
By the definitions of $Y_{1,n}^\tau, Y_{1,n}$ and the elementary
inequality $\log x\le x-1$ for all $x>0$, then we have
$$
\aligned
 &\frac{t(p-1)}{2b_n^2}\log
\ee\exp\left(2\frac{\lambda b_n}{\sqrt
n}\left|Y_{1,n}^\tau-Y_{1,n}\right|\right)\\
\le & \frac{tp}{2b_n^2}\left(\ee\exp\left(2\frac{\lambda b_nm}{\sqrt
n}\left|X_{1,n}^\tau-X_{1,n}\right|\right)-1\right)\\
\le & \frac{tp}{2b_n^2}(e^M-1)\pp\left(|X_{1,n}|>\tau\frac{\sqrt
n}{b_n}\right)+\frac{tp}{2b_n^2}\int^\infty_Me^x\pp\left(|X_{1,n}|\ge\frac{\sqrt
nx}{2\lambda b_nm} \right)dx\\
\le & \frac{tp}{2n\tau^2}(e^M-1)\ee\left(
X_{1,n}^2\mathbb{I}_{\left\{|X_{1,n}|> \tau \frac{\sqrt
 n}{b_n}\right\}}\right)+\frac{n}{2b_n^2m}\int^\infty_Me^x\pp\left(|X_{1,n}|\ge\frac{\sqrt
nx}{2\lambda b_nm} \right)dx\\
 \to & 0
\endaligned
$$
where we used (\ref{Step-1-1}) and the condition (B). Thus the
desired result follows from (\ref{Step-2-1}) and the arbitrariness
of $\lambda$.
\end{proof}
\begin{lem}\label{lem2-miao} Under the assumptions of Lemma \ref{lem-miao}, for any
$r>0$, we have
$$
 \lim_{n\to\infty}\frac{1}{b_n^2}\log\pp\left(\frac{1}{b_n\sqrt
 n}\left|\sum_{j=tp+1}^lY_{j,n}+\sum_{h=1}^tY_{hp,n}+\sum_{i=lm+1}^nX_{i,n}\right|>r\right)=-\infty.
$$
\end{lem}
\begin{proof}
Taking along the lines of the proof of Lemma \ref{lem1-miao}, we can
prove that anyone of the following three terms
$$
\sum_{j=tp+1}^lY_{j,n},\ \ \ \sum_{h=1}^tY_{hp,n}\ \ \text{and}\ \ \
\sum_{i=lm+1}^nX_{i,n}
$$
 can be negligible with respect
to the moderate deviation principle.
\end{proof}
At last, based on Lemma \ref{lem1-miao} and Lemma \ref{lem2-miao},
the proof of Lemma \ref{lem-miao} can be finished.
\end{proof}


\end{document}